\documentclass[11pt,leqno]{amsart}

\usepackage{amssymb}
\usepackage{amsmath}
\usepackage{amsxtra}
\usepackage{amscd}
\usepackage{amsfonts}
\usepackage[utf8]{inputenc}
\usepackage{hyperref}
\usepackage[all]{xy}
\usepackage{tikz}
\usetikzlibrary{arrows, decorations.markings}
\usepackage{float}
\usepackage{bicaption}
\usepackage{graphicx}
\usepackage{epstopdf}

\newtheorem{theorem}{Theorem}
\newtheorem{corollary}{Corollary}
\newtheorem{lemma}{Lemma}
\newtheorem{proposition}{Proposition}
\newtheorem{definition}{Definition}
\newtheorem{remark}{Remark}

\numberwithin{equation}{section}

\makeatletter
\let\@wraptoccontribs\wraptoccontribs
\makeatother

\begin{document} 

\title{A new two-variable generalization of the Jones polynomial}

\author{D. Goundaroulis}
\address{Center for Integrative Genomics,
UNIL,
Batiment G\'enopode, CH-1015 Lausanne, Switzerland.}
\address{Swiss Institute of Bioinformatics, CH-1015 Lausanne, Switzerland.}
\email{dimoklis.gkountaroulis@unil.ch}

\author{S. Lambropoulou}
\address{Department of Mathematics,
National Technical University of Athens,
Zografou campus, GR--157 80 Athens, Greece.}
\email{sofia@math.ntua.gr}
\urladdr{www.math.ntua.gr/~sofia/}




\date{}

\begin{abstract}      
We present a new 2-variable generalization of the Jones polynomial that can be defined through the skein relation of the Jones polynomial. The well-definedness of this  invariant is proved both algebraically and diagrammatically as well as via a closed combinatorial formula. This new invariant is able to distinguish more pairs of non-isotopic links than the original Jones polynomial, such as the Thistlethwaite link from the unlink with two components. 
\end {abstract}

\maketitle

\section{Introduction}

In the last ten years there has been a new spark of interest for polynomial invariants for framed and classical knots and links. One of the concepts that appeared was that of the framization of a knot algebra, which was first proposed by J. Juyumaya and the second author  in \cite{22, 23}. In their original work, they constructed new 2-variable polynomial invariants for framed and classical links  via the Yokonuma--Hecke algebras ${\rm Y}_{d,n}(u)$ \cite{34}, which are quotients of the framed braid group $\mathcal{F}_n$. The algebras ${\rm Y}_{d,n}(u)$ can be considered as framizations of the Iwahori--Hecke algebra of type $A$, ${\rm H}_n(u)$, and for $d=1$, ${\rm Y}_{1,n}(u)$ coincides with ${\rm H}_n(u)$. They used the Juyumaya trace  \cite{18} with parameters $z, x_1, \ldots ,x_{d-1}$ on ${\rm Y}_{d,n}(u)$ and the so-called ${\rm E}$-condition imposed on the framing parameters $x_i$, $1 \leq i \leq d-1$.  These new invariants and especially those for classical links had to be compared to other known invariants like the Homflypt polynomial \cite{30,29,11,31}. The lack of an appropriate topological interpretation, however, prevented for quite a while such a comparison. 

 Eventually, in \cite{4} the implementation of a different presentation  for the algebra ${\rm Y}_{d,n}$, with parameter $q$ instead of $u$, revealed that the corresponding invariants for classical links, $\Theta_d (q, \lambda_D)$, satisfy the  skein relation of the Homflypt polynomial $P$, but only on crossings between different components of the link, and this fact allowed the comparison of the invariants $\Theta_d$ to $P$ \cite{24, 6, 4}. We note that  for $d=1$ the invariant  $\Theta_1$ coincides with $P$. As it turned out, for $d\neq 1$, the invariants $\Theta_d$ are {\it not } topologically equivalent to the Homflypt polynomial {\it on links}, meaning that they distinguish pairs of non-isotopic oriented classical links, which are $P$-equivalent. Moreover, the invariants $\Theta_d$ are not topologically equivalent to the Kauffman polynomial \cite{25}, since on {\it knots} they are topologically equivalent to the Homflypt polynomial. 

Further, the invariants $\Theta_d$ have been generalized in \cite{4} to a new 3-variable invariant $\Theta (q, \lambda, E)$  which is stronger than the Homflypt polynomial. The indeterminate $E$ can be viewed as a shortcut around  the ${\rm E}$-system and the determination of its solutions.  The well-definedness of $\Theta$ was proved in \cite{4} by showing that  on classical links it coincides with the invariant for tied links that is derived from the algebra of braids and ties \cite{1, 3}. 
An alternative, purely skein-theoretical approach, was also given in \cite{4}. A self-contained proof of the well-definedness of $\Theta$ via this approach has been recently given by L. Kauffman and the second author \cite{27}, while in \cite{28} they also construct new state sum models related to $\Theta$. Furthermore, a closed combinatorial formula for the invariant $\Theta$  has been proved by W.B.R. Lickorish in \cite[Appendix~B]{4}, which involves  the values of the Homflypt polynomial on sublinks of a given oriented link.

In \cite{14} the framization of the Temperley--Lieb algebra, ${\rm FTL}_{d,n}(q)$, was introduced as a quotient of the Yokonuma--Hecke algebra. From the algebras ${\rm FTL}_{d,n}(q)$ one-variable invariants for oriented classical links, $\theta_d (q)$, were derived by finding the necessary and sufficient conditions for the Juyumaya trace to pass through to the quotient algebra. The invariants $\theta_d$, being specializations of the invariants $\Theta_d$, they carry all of their properties. Note that, for $d=1$, $\theta_1$ coincides with the Jones polynomial  \cite{{16}}. Further, as it was proved  in \cite{14}, for $d\neq1$, $\theta_d$ is {\it not } topologically equivalent to the Jones polynomial {\it on links}.

\smallbreak 

In this paper we introduce a 2-variable isotopy invariant for oriented classical knots and links, denoted by $\theta$, which generalizes the invariants $\theta_d$ and is a specialization of $\Theta$. More precisely:
\begin{theorem} \label{mainthm}
Let $q, E$ be indeterminates. There exists a unique ambient isotopy link invariant defined on the set of classical oriented link diagrams
\[
\theta : \mathcal{L} \rightarrow \mathbb{C}[q^{\pm 1} , E^{\pm 1}]
\]
 by the following rules:

\begin{enumerate}
\item On crossings between different components the following skein relation holds:
\[
q^{-2}\, \theta (  \raisebox{-.1cm}{\begin{tikzpicture}[scale=.2]
\draw [line width=0.35mm, draw=red]  (-1,-1)-- (-0.22,-0.22);
\draw  [line width=0.35mm, draw=black](-1,1)--(0,0);
\draw  [line width=0.35mm, draw=red] (0.22,0.22) -- (1,1)[->];
\draw [line width=0.35mm, draw=black]   (0,0) -- +(1,-1)[->];
\end{tikzpicture}}\, ) - q^2\, \theta (\raisebox{-.1cm}{\begin{tikzpicture}[scale=.2]
\draw  [line width=0.35mm, draw=red] (-1,-1)-- (0,0) ;
\draw [line width=0.35mm, draw=black] (-1,1)--(-0.22,0.22);
\draw [line width=0.35mm, draw=red] (0,0) -- (1,1)[->];
\draw [line width=0.35mm, draw=black]   (0.22,-0.22) -- +(.8,-.8)[->]; \end{tikzpicture}}\,) = (q - q^{-1})\, \theta (\raisebox{-.1cm}{\begin{tikzpicture}[scale=.2, mydeco/.style = {decoration = {markings, 
                                                       mark = at position #1 with {\arrow{>}}}
                                       }]
\draw [red, line width=0.35mm, postaction = {mydeco=.6 ,decorate}] plot [smooth, tension=2] coordinates { (-1,.8) (0, 0.5) (1,.8)};
\draw [red, line width=0.35mm, postaction = {mydeco=.6 ,decorate}] plot [smooth, tension=2] coordinates { (-1,-.8) (0, -0.5) (1,-.8)};
\end{tikzpicture}}\,).
\]
\item For a  union $\mathcal{K} = \sqcup_{i=1}^r K_i$ of $r$ unlinked knots,  $r \geq 1$, it holds that:
\[
\theta (\mathcal{K}) = E^{1-r}V(\mathcal{K})
\]
where $V(\mathcal{K})$ is the value of the Jones polynomial on $\mathcal{K}$.
\end{enumerate}
\end{theorem}

We establish the invariant $\theta(q,E)$ both algebraically and skein-theoretically,  resting on the analogous proofs for the invariant $\Theta (q, \lambda, E)$ \cite{4,27}. Algebraically, we first construct a new 2-variable invariant for tied links,  ${\overline \theta}(q,E)$, using the Markov trace $\rho$ on the algebra of braids and ties $\mathcal{E}_n(q)$ \cite{1} and specifying the necessary and sufficient conditions so that $\rho$ passes to the quotient algebra ${\rm PTL}_n(q)$, the partition Temperley--Lieb algebra \cite{19}. By restricting ${\overline \theta}(q,E)$ to classical links we obtain the invariant $\theta(q,E)$. A second algebraic way to obtain the invariant $\theta$ is by considering an isomorphism between the algebra ${\rm PTL}_{n}(q)$ and the  subalgebra of the algebra ${\rm FTL}_{d,n}(q)$ generated only by the braiding generators \cite{15}.   Skein-theoretically,  $\theta$ can be described as a specialization of $\Theta$ via the skein relation of the Jones polynomial for oriented links and a set of initial conditions. We continue by giving a closed combinatorial formula for  $\theta$ derived from the formula of W.B.R. Lickorish for  $\Theta$ \cite[Appendix~B]{4}. Namely, for an $n$-component oriented link $L$ we have:
\begin{equation}\label{thetaformula}
\theta(q,E)(L) = \sum_{k=1}^n (-1)^{k-1} (q+q^{-1})^{k-1} E_k \sum_{\pi} q^{4\nu(\pi)} V(\pi L), 
\end{equation}
where the second summation is over all partitions  $\pi$ of the components of $L$ into $k$ (unordered) subsets and $V(\pi L)$ denotes the product of the Jones polynomials of the $k$ sublinks of $L$ defined by $\pi$.  Furthermore, $\nu(\pi)$ is the sum of all linking numbers of pairs of components of $L$ that are distinct sets of $\pi$ and $E_k := (E^{-1} -1 ) (E^{-1} -2 ) \ldots (E^{-1} - k+1)$, with $E_1=1$. 
 Finally, we relate  $\theta$ to the oriented extension of the Kauffman bracket polynomial and we show that it distinguishes the Thistlethwaite link from the unlink. This last result is due to the first author and Louis H. Kauffman. Although the invariant $\theta$ comes as a specialization of  $\Theta$ \cite{4}, yet it is `new' in the sense that it does not coincide with any other known generalization of the Jones polynomial. To recapitulate, the invariant $\theta(q,E)$ is topologically equivalent to the Jones polynomial on knots and it is stronger than the Jones polynomial on links. Consequently, it is not topologically equivalent to the Homflypt or to the Kauffman polynomial.

 For a better understanding of the relations between the various algebras as well as their associated invariants, we refer the reader to Figure~\ref{diag}. The outer arrows indicate the algebras involved and the algebra projection maps, while the dotted arrows point from the algebras to their associated invariants. The inner part of the diagram shows the inclusions and the specializations of the invariants.

\begin{figure}[!tbh] 
{ \raggedleft
\[
\centerline{ \xymatrix{ & & \txt{Hecke\ algebra}\ar@{.>}[d]  & & \\
&  & \txt{ Homflypt} \ar@{_{(}->}[d]<-.5ex>& &\\
&\txt{Algebra of\\ braids and ties}\ar@/^2.0pc/[uur] \ar@{->>}[d] \ar@{.>}[r] & \Theta(q,\lambda, E)  \ar@{->}[u]<-.5ex>_{E=1}\ar[d]^{\lambda= q^4}  \ar@{->}[r]<-.5ex>_{E=1/d} & \Theta_d(q,\lambda_d) \ar@{_{(}->}[l]<-.5ex>\ar[d]^{\lambda_d=q^4}  \ar@/_/[ul]^{d=1} & \txt{Yokonuma-Hecke\\algebra} \ar@/_2.0pc/[uull]^{d=1}\ar@{->>}[d] \ar@{.>}[l]  \\
&\txt{Partition\\Temperley-Lieb\\algebra} \ar@/_2.0pc/[ddr] \ar@{.>}[r] & \theta(q, E)  \ar@{->}[d]<-.5ex>_{E=1}  \ar@{->}[r] <-.5ex>_{E=1/d}& \theta_d(q) \ar@/^/[dl]^{d=1}  \ar@{_{(}->}[l] <-.5ex> & \txt{Framization of\\ Temperley-Lieb\\algebra} \ar@{.>}[l] \ar@/^2.0pc/[ddll]_{d=1}\\
&  & \txt{Jones polynomial} \ar@{_{(}->}[u]<-.5ex>& &\\
& & \txt{Temperley-Lieb\\algebra}\ar@{.>}[u] & &} }
\]
}
\caption{An overview of the relations between the algebras and their derived invariants.}\label{diag}

\end{figure}
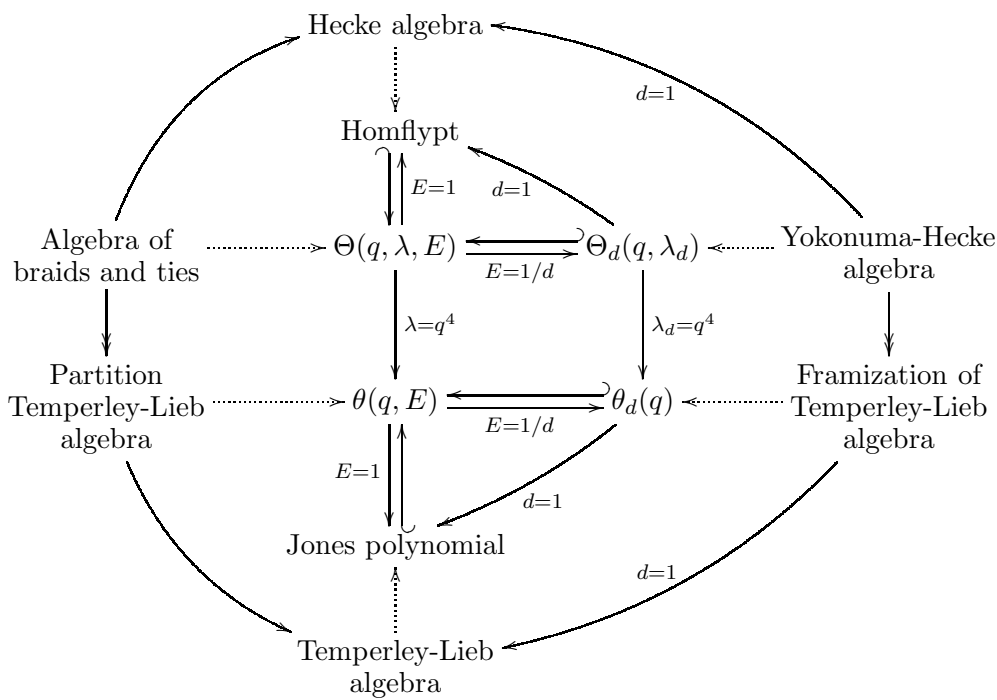

\smallbreak
The outline of the paper is as follows:
In Section~\ref{notation} we set the basic notations.  Section~\ref{prelim} is dedicated to providing necessary definitions and results, including:  some basic facts about the Yokonuma--Hecke algebra, the Juyumaya trace ${\rm tr}_d$, and the solutions of the ${\rm E}$-system.  We then describe the construction of the invariants  $\Theta_{d}$ for classical links through the use of the specialized trace ${\rm tr}_{d,D}$. In Section~\ref{ftlsec} we recall the definition of the framization of the Temperley--Lieb algebra ${\rm FTL}_{d,n}(q)$ and we present briefly the results of \cite{14} regarding the 1-variable invariants $\theta_d$ for classical links derived from ${\rm FTL}_{d,n}(q)$. In Section~\ref{btsec} we present the algebra of braids and ties and we describe the different methods for generalizing the invariants $\Theta_d$ to the 3-variable invariant $\Theta(q,\lambda, E)$, following \cite{4}. In Section~\ref{ptlsec} we present the main result of this paper, the generalization of the invariants $\theta_d(q)$ to the  2-variable invariant $\theta(q, E)$ and we show that it is stronger than the Jones polynomial. We define $\theta$ algebraically as well as diagrammatically and we present the closed combinatorial formula for the invariant $\theta$.

\section{Notations} \label{notation}

We fix two positive integers, $d$ and $n$. Every algebra considered in this paper is an associative unital algebra over the field $\mathbb{C}(q)$, where $q$ is an indeterminate.   The {\it framed braid group} on $n$ strands is defined as the semi-direct product of the Artin braid group $B_n$ with $n$ copies of $\mathbb{Z}$, ${\mathcal F}_{n} = \mathbb{Z}^n \rtimes  B_n$, where the action of the braid group $B_n$ on $\mathbb{Z}^n$ is given by the permutation induced by a braid on the indices $\sigma_it_j=t_{s_i(j)}\sigma_i$. Topologically, $t_i$ means framing 1 on the $i$-th strand of a braid and so the generators $t_i$ of $\mathbb{Z}^n$ are called the {\it framing generators}.  The {\it modular framed braid group}, $\mathcal{F}_{d,n} = (\mathbb{Z}/d\mathbb{Z})^n \rtimes B_n$, is defined by taking the framings modulo $d$. Due to the defining action above, a word $w$ in ${\mathcal F}_{n}$ (resp. $\mathcal{F}_{d,n}$) has the {\it splitting property}, that is, it splits into the  {\it framing part} and the {\it braiding  part}:
$
w = t_1^{a_1}\ldots t_n^{a_n} \, \sigma,
$
where $\sigma \in B_n$ and $a_i \in \mathbb{Z}$ (resp. $\mathbb{Z}/d\mathbb{Z}$). So $w$ is a classical braid with an integer (resp. an integer modulo $d$) attached to each
strand. 

\section{The invariants $\Theta_d$ and the Yokonuma--Hecke algebra}\label{prelim}

The {\it Yokonuma--Hecke algebra} of type $A$ ${\rm Y}_{d,n}(q)$ \cite{34} is defined as the quotient of the group algebra
$\mathbb{C}(q) {\mathcal F}_{d,n}$ over the two-sided ideal  generated by the elements
$
\sigma_i^2 - 1 -  (q-q^{-1}) \, e_i \, \sigma_i$ for all  $i,$ 
where $e_{i} := \frac{1}{d} \sum_{s=0}^{d-1}t_i^s t_{i+1}^{d-s}$, for  $i=1,\ldots , n-1$. The elements $e_i$ in $\mathcal{F}_{d,n}$ are idempotents \cite{18,20} and can be generalized  to the elements $e_{i,j} := \frac{1}{d} \sum_{s=0}^{d-1}t_i^s t_{j}^{d-s}$, for any indices $i,j$. We also define {\it the shift of $e_i$ by $m$}, $e_i^{(m)} := \frac{1}{d} \sum_{s=0}^{d-1} t_i^{m+s} t_{i+1}^{d-s}$, for any $0\leq m \leq d-1$. The generators of the ideal give rise to the following quadratic relations in ${\rm Y}_{d,n}(q)$: 
\begin{equation}\label{yqeq}
g_i^2 = 1 + (q-q^{-1}) \, e_i \, g_i,
\end{equation}
where $g_i$ corresponds to $\sigma_i$ and $t_j$ to $t_j$. Moreover, \eqref{yqeq} implies that the elements $g_i$ are invertible with $g_i^{-1} = g_i - (q- q^{-1}) e_i$, $ 1 \leq i \leq n-1$. The $t_j$'s are the {\it framing generators} while the $g_i$'s are the {\it braiding generators} of ${\rm Y}_{d,n}(q)$. By its construction, the Yokonuma--Hecke algebra of type $A$ is considered as {\it the framization of the Iwahori-Hecke algebra of type} A.

\subsection{{\it Markov trace on ${\rm Y}_{d,n}(q)$}} In \cite{18} Juyumaya has proven that ${\rm Y}_{d,n}(q)$ supports a unique Markov trace function  ${\rm tr}_d:  \cup_{n=1}^{\infty}{\rm Y}_{d,n}(q) \longrightarrow    \mathbb{C}(q)[z, x_1, \ldots, x_{d-1}],$ where $z$, $x_1$, $\ldots, x_{d-1}$ are indeterminates, defined inductively on $n$ by the following rules:
\[
\begin{array}{rcll}
{\rm tr}_d(ab) & = & {\rm tr}_d(ba)  \qquad &  \\
{\rm tr}_d(1) & = & 1 & \\
{\rm tr}_d(ag_n) & = & z\, {\rm tr}_d(a) \qquad &  (\text{Markov  property} )\\
{\rm tr}_d(at_{n+1}^s) & = & x_s {\rm tr}_d(a)\qquad  & (  s = 1, \ldots , d-1) ,
\end{array}
\]
where  $a,b \in {\rm Y}_{d,n}(q)$. Using the rules of ${\rm tr}_d$ and setting  $x_0:=1$, one deduces that ${\rm tr}_d(e_i)$ takes the same value for all $i$, that is: ${\rm tr}_d(e_i)= \frac{1}{d}\sum_{s=0}^{d-1}x_{s}x_{d-s}:=E.$ Moreover, we also define $E^{(m)} :={\rm tr}_d(e_i^{(m)})= \frac{1}{d}\sum_{s=0}^{d-1}x_{m+s}x_{d-s}$ ($0\leq m \leq d-1$), {\it the shift by $m$ of $E$} . Notice that  $E^{(0)} = E$.

\subsection{{\it The ${\rm E}$-system}} In order to define framed and classical link invariants via the trace ${\rm tr}_d$, one should re-scale ${\rm tr}_d$ so that ${\rm tr}_d(ag_n) = {\rm tr}_d (a g_n^{-1})$, $a \, \in \, {\rm Y}_{d,n}(q)$. Unfortunately, the trace ${\rm tr}_d$ is the only known Markov trace that does not re-scale directly \cite{22}. The {\it ${\rm E}$-system} is the following system of non-linear equations
\begin{equation}\label{Esystem}
E^{(m)}  =   x_{m}E \qquad (1\leq m \leq d-1).
\end{equation}
 that was introduced in order to find the  necessary and sufficient conditions that needed to be applied on the parameters $x_i$  of  ${\rm tr}_d$ so that re-scaling would be possible \cite{22}. We say that the $(d-1)$-tuple  of complex numbers $({\rm x}_1, \ldots , {\rm x}_{d-1})$ satisfies the {\it ${\rm E}$-condition} if ${\rm x}_1,\ldots , {\rm x}_{d-1}$ are solutions of the ${\rm E}$-system. The full set of solutions of the ${\rm E}$-system is given by Paul G\'{e}rardin \cite[Appendix]{22} using tools of harmonic analysis on finite groups and they are parametrized by the non-empty subsets of $\mathbb{Z}/d\mathbb{Z}$, for details see \cite{22, 14}. More precisely, he interpreted the solution 
{$({\rm x}_1,\ldots,{\rm x}_{d-1})$} 
of the ${\rm E}$-system, as the complex function $x : \mathbb{Z}/d\mathbb{Z} \rightarrow \mathbb{C}$ that sends $k \mapsto {\rm x}_k$, $k\neq 0$ and $0 \mapsto 1$. Now let $\chi_m$ denote the character of the group $\mathbb{Z}/d\mathbb{Z}$ and denote by $\mathbf{i}_m:=\sum_{s=0}^{d-1} \chi_m( s) t^s$, for $ m\in {\mathbb Z}/d{\mathbb Z} \in \mathbb{C}[\mathbb{Z}/d\mathbb{Z}]$. We then have that the solutions of the ${\rm E}$-system are of the following form:
\[
{\rm x}_s = \frac{1}{|D|} \sum_{m \in D} \mathbf{i}_m (s), \quad 1 \leq s \leq d-1
\]
where  $D$ is a non-empty subset of $\mathbb{Z}/d\mathbb{Z}$. Hence, the solutions of the ${\rm E}$-system are parametrized by the non-empty subsets of $\mathbb{Z}/d\mathbb{Z}$. Two obvious solutions of the ${\rm E}$-system are  when all $x_i$'s take the value zero and when the $x_i$'s are specialized to the $d$-th roots of unity. For the rest of the paper we fix $X_D = \left ( {\rm x}_1 , \ldots , {\rm x}_{d-1} \right )$ to be a solution of the ${\rm E}$-system parametrized by the non-empty subset $D$ of $\mathbb{Z}/d\mathbb{Z}$. If we specialize the trace parameters $x_i$ of ${\rm tr}_d$ to the values ${\rm x}_i$  we obtain the {\it specialized trace} ${\rm tr}_{d,D}$ with parameter $z$ \cite{5, 4}.  
\subsection{{\it Classical link invariants from ${\rm Y}_{d,n}(q)$}}\label{brsec}
In order to obtain classical link invariants we first map the classical braid group $B_n$ to the algebra ${\rm Y}_{d,n}(q)$ and then we apply ${\rm tr}_{d,D}$. Note that via this mapping the framing generators of the algebra lose any topological context and they are treated simply as formal generators. More precisely, let $\delta$ be the natural algebra homomorphism $\mathbb{C}(q)B_n \longrightarrow {\rm Y}_{d,n}(q)$ that sends $\sigma_i \mapsto g_i$ and let $ {\rm Y}_{d,n}^{\rm (br)}(q):=\delta(\mathbb{C}(q)B_n)$.  It can be easily shown that the subalgebra of ${\rm Y}_{d,n}(q)  $ that is generated by the elements $g_i$ and $e_j$  ($1 \leq i,j \leq n-1$) coincides with the subalgebra ${\rm Y}_{d,n}^{\rm (br)}(q)$   \cite[Remark~4.2]{4}. Further, we note that the framing generators $t_i$ appear in the computation of the specialized trace ${\rm tr}_{d,D}$ on any $\alpha \in B_n$ only after applying the quadratic relation or the inverse relation  and only through the idempotents $e_i$. Hence, in this case and by the ${\rm E}$-condition, the last rule of  the specialized trace: ${\rm tr}_{d,D}(at_{n+1}^s) =  {\rm x}_s {\rm tr}_{d,D}(a)$, for $s = 1, \ldots , d-1$,  can be substituted by the following two rules \cite[Theorem~4.3]{4}:
 \[
 {\rm tr}_{d,D}(a e_n) = E_D \, {\rm tr}_{d,D}(a) \quad \mbox{and} \quad {\rm tr}_{d,D}(a e_n g_n) = z\, {\rm tr}_{d,D}(a).
 \]
Further, if $|D| = m$ then $E_D:= {\rm tr}_{d,D} (e_i) = 1/m$, for all $1 \leq i \leq n-1$ \cite{21, 18}. Let now $\lambda_D := \frac{z- (q-q^{-1}) E_D}{z}$. We then have the following 2-variable invariant for classical knots and links, which is denoted by $\Theta_{d,D}$ \cite[Theorem~3.1]{4}:
\begin{equation}\label{thetainv}
\Theta_{d,D}(q, \lambda_D)(\widehat{\alpha}) := \left ( \frac{1 - \lambda_D}{\sqrt{\lambda_D} (q-q^{-1})E_D} \right)^{n-1} \left ( \sqrt{\lambda_D} \right)^{\varepsilon(\alpha)} {\rm tr}_{d,D}(\delta ( \alpha ) ) ,
\end{equation}
where $\widehat{\alpha}$ is the closure of the framed braid $\alpha$ and $\varepsilon ( \alpha)$ is the algebraic sum of the exponents of the braiding generators $g_i$ in the braid word $\alpha$.  As proved in \cite[Proposition~4.6]{4} the invariants $\Theta_{d,D}$ do not depend on the sets $D$ that parametrize the solutions of the ${\rm E}$-system, but only on their cardinality, meaning that they are parametrized only by the natural numbers. For this reason, we will always consider that $D=\mathbb{Z}/d\mathbb{Z}$, implying that $E_D = 1/d$ and thus the notation of the invariants $\Theta_{d,D}$ shall be simplified to $\Theta_d$ \cite{4}.  It is worth noting that for $d=1$ the algebra ${\rm Y}_{1,n}(q)$ coincides with the Iwahori-Hecke algebra ${\rm H}_n(q)$ and so the invariant $\Theta_{1,\{0\}}$ coincides with the Homflypt polynomial $P$. Furthermore, as proven in \cite{5} the invariants $\Theta_d$ coincide with the Homflypt polynomial also for the trivial cases $q=1$ and $E_D=1$.  

The invariants $\Theta_d$ are topologically equivalent to $P$ on knots and on  unions of unlinked  knots \cite[Theorem~5.8]{4}. However, they are {\it not topologically equivalent to the Homflypt polynomial  for the case of links for $d\geq 2$} \cite[Theorem~7.3]{4} .  

A very interesting property of the invariants $\Theta_d$ is the fact that they satisfy the well-known skein relation of the Homflypt polynomial but only for {\it mixed crossings}, that is, crossings between different components of the link.  In fact, $\Theta_d$ can be defined through this mixed skein relation together with the conditions that for a  union  $\mathcal{K} = \sqcup _{i=1}^r K_i$ of $r$ disjoint knots, with $r\geq 1$, it holds:
\begin{equation}\label{thetaknot}
\Theta_d (\mathcal{K})= E_D^{1-r} P(\mathcal{K}),
\end{equation}
where $P(\mathcal{K})$ is the value of the Homflypt polynomial on $\mathcal{K}$ \cite[Proposition~6.8 and Theorem~6.2]{4}.

So one can evaluate $\Theta_d$ on a link $L $ by using the skein relation in order to unlink the components of $L$ one by one and then apply \eqref{thetaknot} on the resulting sums of  unions of unlinked knots.

\section{The invariants $\theta_d$ and the Framization of the Temperley--Lieb algebra}\label{ftlsec}

The Framization of the Temperley--Lieb algebra was introduced in \cite{14} as the Temperley--Lieb analogue for the Yokonuma--Hecke algebra with the scope of constructing 1-variable polynomial invariants for framed and classical knots and links (see also \cite{12,13}).

For $n \geq 3$, the {\it Framization of the Temperley--Lieb algebra}, denoted  by ${\rm FTL}_{d,n}(q)$, is defined as the quotient of ${\rm Y}_{d,n}(q)$ over the two-sided ideal generated by the element \cite[Definition~5 and Corollary~1]{14}:
\begin{equation}\label{newftl}
 r_{1,2} := e_1 e_2 \Big( 1 + q(g_1 + g_2) + q^2 (g_1 g_2 + g_2 g_1) + q^3 g_1 g_2 g_1 \Big).
 \end{equation}

One of the challenges that emerged was the determination of the necessary and sufficient conditions so that the trace ${\rm tr}_d$ factors through to the quotient algebra ${\rm FTL}_{d,n}(q)$. Indeed:
\begin{theorem}[{\cite[Theorem~6]{14}}]\label{ftlthmgen}
The trace ${\rm tr}_d$ passes to ${\rm FTL}_{d,n}(q)$ if and only if the parameters of the trace satisfy: 
\begin{align}\label{newqvalx}
x_k = -qz \left(\sum_{m\in {\rm Sup}_1}\chi_{m}(k) + (q^2+1)\sum_{m\in {\rm Sup}_2}\chi_{m}(k) \right),
\end{align}
\begin{align}\label{newqvalz}
 z=-\frac{1}{q\vert {\rm Sup_1}\vert + q(q^2+1)\vert {\rm Sup_2}\vert  }.
\end{align}
where ${\rm Sup}_1\sqcup \rm{Sup}_2$ (disjoint union) is the support of the Fourier transform of $x$, and $x$ is the complex function on $\mathbb{Z}/d\mathbb{Z}$,   
that maps $0$ to $1$ and $k$ to the trace parameter  $x_k$.
\end{theorem}
We note that \eqref{newqvalx} includes all solutions of the ${\rm E}$-system.  These solutions can be recovered by simply letting either ${\rm Sup}_1 = \emptyset$ or ${\rm Sup}_2 = \emptyset$ \cite[Corollary~3]{14}. Then, the trace parameter $z$ takes the value $z= - q^{-1}E_D / (q^2+1)$ or $ z= q^{-1} E_D$ respectively \cite[Section~7]{14}. The value $z = -q^{-1} E_D$ is discarded since it is of no topological importance \cite[Remark~10]{14}. Specializing now  $z= - q^{-1}E_D/(q^2+1)$ in \eqref{thetainv} we obtain the following 1-variable invariant for classical knots and links.
\begin{equation}\label{onevarinv}
\theta_d(q)(\widehat{\alpha}) := \left ( - \frac{1+q^2}{qE_D} \right )^{n-1} q^{2 \varepsilon(\alpha)} {\rm tr}_{d,D}\left ( \delta ( \alpha) \right ) = \Theta_d (q, q^4 ) (\widehat{\alpha}),
\end{equation}
where $\varepsilon(\alpha)$ and $\delta$ as in \eqref{thetainv}. 

\begin{remark}\label{remspec} \rm
As mentioned in the previous section the invariants $\Theta_d$ are topologically equivalent to the Homflypt polynomial $P$ on knots and  unions of unlinked knots, while they are not topologically equivalent to the polynomial $P$ for the case of links. For $z= - q^{-1}E_D/ (q^2+1)$ these properties carry through to the invariants $\theta_d$ when compared to the Jones polynomial $V$. Furthermore, $\Theta_d$ and $\theta_d$ have analogous properties to the Homflypt polynomial, see \cite{6}.
\end{remark}
For $d\in \mathbb{Z}_{>1}$, the invariants $\theta_d(q)$  for classical links are not topologically equivalent to the Jones polynomial. Further, the invariants $\theta_d(q)$ satisfy the following special skein relation:
$
q^{-2} \theta_d ( L_+) - q^2 \theta_d(L_-) = (q - q^{-1} ) \theta_d(L_0),
$
where $L_+$, $L_-$, $L_0$ constitute an oriented Conway triple involving different components  \cite[Theorem~9]{14}.

\section{The invariant $\Theta$ and the algebra of braids and ties}\label{btsec}

The aim of this paper is to introduce a new 2-variable generalization of the Jones polynomial using the invariants $\theta_d$ (see next section). The cornerstone of this construction is the 3-variable generalization $\Theta$ of the invariants $\Theta_d$ constructed in \cite{4} and which will be presented here. The well-definedness of $\Theta$ has been proved both algebraically \cite{4} and skein-theoretically \cite{27,28}. We will present all approaches, as they are all very informative.

\subsection{{\it $\Theta$ via the algebra of braids and ties $\mathcal{E}_n(q)$}} We start introducing the so-called {\it algebra of braids and ties} $\mathcal{E}_n(q)$ that was first introduced by Juyumaya and Aicardi in \cite{1}. The algebra $\mathcal{E}_n(q)$ is the $\mathbb{C}(q)$-algebra that is generated by the elements $b_1, \ldots, b_{n-1}, \epsilon_1,$ $\ldots,\epsilon_{n-1}$ that satisfy the following relations:
\begin{equation}\label{btpres}
\begin{aligned}
b_i b_{i+1} b_i &= b_{i+1} b_i b_{i+1}&\\
b_i b_j &= b_j b_i \quad &\mbox{for }  |i-j|>1\\
\epsilon_i \epsilon_j &=\epsilon_j \epsilon_i \quad &\mbox{for } |i-j|>1\\
\epsilon_i^2 &= \epsilon_i&\\
\epsilon_ib_i &= b_i \epsilon_i&\\
\epsilon_i b_j &= b_j\epsilon_i \quad &\mbox{for } |i-j|>1\\
\epsilon_i\epsilon_j b_i &= b_i \epsilon_i \epsilon_j  = \epsilon_j b_i \epsilon_j \quad &\mbox{for } |i-j|=1\\
\epsilon_i b_j b_i &= b_j b_i \epsilon_j \quad &\mbox{for } |i-j|=1\\
b_i^2& =  1 +(q-q^{-1})\epsilon_ib_i &
\end{aligned}
\end{equation}

\begin{remark}\label{btswitch} \rm
Originally the algebra of braids and ties was introduced by in \cite{1} with a presentation using generators $\widetilde{b}_1, \ldots, \widetilde{b}_{n-1}$, $\epsilon_1, \ldots , \epsilon_{n-1}$ that satisfy all relations in the presentation \eqref{btpres} except for the quadratic relation which is replaced by one  with parameter $u$ instead of $q$, namely:
\[
(\widetilde{b}_i)^2 = 1 + (u-1)\epsilon_i + (u-1)\epsilon_i \widetilde{b}_i
\]
The isomorphic algebra $\mathcal{E}_n(q)$ was introduced in \cite{4}. The algebra isomorphism is the transformation induced by $b_i := \widetilde{b}_i +(q^{-1} -1) \epsilon_i \widetilde{b}_i$ (or equivalently $\widetilde{b}_i := b_i +(q-1) \epsilon_i b_i$) and choosing $u=q^2$.
\end{remark}

Denote now by $\mathbf{n}$ the set $\{1, \ldots , n \}$ and by $P(n)$ the set of all partitions of $\mathbf{n}$. Let $I \in P(n)$ and let  $I_j \in I$ with $I_j = \{ i_1 ,\ldots , i_m \}$. We then define the following elements in $\mathcal{E}_n(q)$: 
\[\epsilon_{I_j} := \epsilon_{i_1, i_2}\epsilon_{i_2, i_3} \ldots \epsilon_{i_{m-1}, i_m},
\] 
where $\epsilon_{i, j} = b_i\ldots b_{j-2}\epsilon_{j-1}b_{j-2}^{-1} \ldots b_i^{-1}$, for $1 \leq i < j \leq n$. Notice that with this notation we have that $\epsilon_{i, i+1} = \epsilon_i$. Let now  $I = \{ I_1, \ldots , I_r \}$ be a partition of $\mathbf{n}$. Using the above notation we further define:
\[
\epsilon_I = \prod_{k} \epsilon_{I_k}
\]
Then the following set is a canonical basis for the algebra $\mathcal{E}_n(q)$ \cite{32}:\small
\[
\mathfrak{B}_{\mathcal{E}_n(q)} = \left \{ \epsilon_I (b_{i_1} \ldots b_{k_1} ) \ldots ( b_{i_p} \ldots b_{k_p} ) \, \vert \, I\in P(n), \ 1 \leq i_1 < \ldots < i_p \leq n-1, \ k_j \leq i_j , \ 1 \leq j \leq n-1\right \}.
\]\normalsize
The dimension of the algebra $\mathcal{E}_n(q)$ is ${\rm dim}_{\mathbb{C}(q)} \mathcal{E}_n(q) = \beta_n \,n !$, where $\beta_n$ is the $n$-th Bell number.
For example for $n=3$, we have ${\rm dim}_{\mathbb{C}(q)} \mathcal{E}_3(q) = 30$ and the following set is a basis for $\mathcal{E}_3(q)$:
\begin{equation}\label{e3qbas}
\mathfrak{B}_{\mathcal{E}_3(q)} =\left \{ \epsilon_I, \epsilon_I b_1, \epsilon_I b_2, \epsilon_I b_1 b_2 ,\epsilon_I b_2 b_1 , \epsilon_I b_1 b_2 b_1  \ \vert  \ I \in P(3) \right \}.
\end{equation}
Moreover, the algebra $\mathcal{E}_n(q)$ supports a unique Markov trace $\rho: \bigcup_{n \geq 0} \mathcal{E}_n(q) \rightarrow \mathbb{C}[q^{\pm 1} , z^{\pm 1} , E^{\pm 1}]$ that can be defined using the following rules \cite[Theorem~3]{2}:
\[
\begin{array}{crcll}
(i) & \rho(ab) &= &\rho (ba) & a,b \in \mathcal{E}_n(q)\\
(ii) & \rho(1) & = & 1 & \\
(iii) & \rho ( a b_n) & =& z \rho (a) & a \in \mathcal{E}_n(q)\\
(iv) & \rho (a \epsilon_n) & = & E \rho (a) &a \in \mathcal{E}_n(q)\\
(v) & \rho (a \epsilon_n b_n) & = & z \rho( a) & a \in \mathcal{E}_n(q)
\end{array}
\]
Notice now the resemblance of the rules of the trace $\rho$ with the rules of the specialized trace ${\rm tr}_{d,D}$ and recall the discussion of Section~\ref{brsec} regarding the subalgebra ${\rm Y}_{d,n}^{({\rm br})}(q)$.
\smallbreak

 The Markov trace $\rho$ gives rise to a 3-variable invariant $\overline\Theta$ of {\it tied links} \cite{4}. The concept of a tied link was first introduced by Juyumaya and Aicardi in \cite{3}. A tied link is a classical link together with a set of ties, that is, a set containing unordered pairs of points that belong to the components of the link. The ties can be seen as springs that can slide along the component(s) that they connect and can cross freely any arc of the diagram. The components that are connected with a tie are not necessarily distinct. If two ties join the same two components, one of them can be removed, and any tie on a single component can be also removed. A tie that cannot be removed is called {\it essential}. A tied link is obtained by closing a {\it tied braid}. The {\it tied braid monoid} $TB_n$ (defined in \cite{3}) is generated by the braiding generators $\sigma_1, \ldots , \sigma_{n-1}$ and the generating ties $\eta_1 , \ldots , \eta_{n-1}$, where $\eta_i$ connects the $i$-th with the $(i+1)$-st strand of a tied braid. The algebra $\mathcal{E}_n(q)$ is a quotient of the algebra $\mathbb{C}(q)TB_n$ over the ideal generated by the quadratic relations for the $b_i$'s, recall \eqref{btpres}. Denote by $\pi : \mathbb{C}(q)TB_n \longrightarrow \mathcal{E}_n(q)$ the natural surjection defined by $\sigma_i \mapsto b_i$ and $\eta_i \mapsto \epsilon_i$ and let:
 \begin{equation}\label{lambdaval}
 \lambda = \frac{z - (q-q^{-1}) E}{z}.
 \end{equation}
  The invariant $\overline \Theta$ is defined as follows \cite{4}:
\begin{equation}\label{ovthet}
\overline{\Theta}(q, \lambda, E) (\widehat{\alpha}) = \left( \frac{1 - \lambda}{\sqrt{\lambda} (q-q^{-1}) E} \right )^{n-1} \sqrt{\lambda}^{\varepsilon(\alpha)} \rho(\pi(\alpha)),
\end{equation}
where $\alpha \in TB_n$ and $\varepsilon(\alpha)$ as in \eqref{thetainv}. We note that for $E=1$ the invariant $\overline\Theta$ coincides with the Homflypt polynomial.  Furthermore, the invariant $\overline\Theta$ restricts to a 3-variable invariant of classical links denoted by $\Theta$, namely:
\begin{equation}\label{thetacl}
\Theta(q, \lambda, E) (\widehat{\alpha}) = \left( \frac{1 - \lambda}{\sqrt{\lambda} (q-q^{-1}) E} \right )^{n-1} \sqrt{\lambda}^{\varepsilon(\alpha)} \rho(\overline\pi(\alpha)),
\end{equation}
where  $\alpha \in B_n$, $\varepsilon(\alpha)$ as in \eqref{ovthet} and  $\overline\pi : \mathbb{C}B_n \longrightarrow \mathcal{E}_n(q)$ the algebra homomorphism that sends $\sigma_i \mapsto b_i$.
We note that, originally, an invariant for tied links $\overline\Delta$ was defined in \cite{3}.
\begin{remark}\label{thetcoinc} \rm
For $E = 1/d$, with $d \in \mathbb{N}$, the invariant $\Theta(q, \lambda , 1/d)$ coincides with the invariant $\Theta_d (q , \lambda_d)$.
\end{remark}

\subsection{\it $\Theta$ as a generalization of the invariants $\Theta_d$}\label{thetaclrem}

Recall the discussion in Section~\ref{brsec} regarding the subalgebra ${\rm Y}^{\rm (br)}_{d,n}(q)$. Let $\phi$ be the mapping $\phi : \mathcal{E}_n(q) \longrightarrow {\rm Y}_{d,n}(q)$ that sends $b_i \mapsto g_i$ and $\epsilon_i \mapsto e_i$. In \cite[Theorem~8]{10} it was proven that the map $\phi$ is an embedding for $d\geq n$. We thus have the following result \cite{10,4, 15}:

\begin{proposition}\label{isoprop}
For $d \geq n$, the algebra of braids and ties $\mathcal{E}_n(q)$ is isomorphic to the subalgebra ${\rm Y}_{d,n}^{({\rm br})}(q)$ of ${\rm Y}_{d,n}(q)$.
\end{proposition}

From Proposition~\ref{isoprop},  if  $E_D = 1/d$ is considered as an indeterminate in the rules of the specialized trace ${\rm tr}_{d,D}$ restricted to ${\rm Y}^{(\rm br)}_{d,n}(q)$, then it is well-defined since it coincides with the trace $\rho$ on $\mathcal{E}_n(q)$ and, therefore, the invariant $\Theta$ can be constructed directly through ${\rm Y}_{d,n}^{(\rm br)}(q)$.

\subsection{\it $\Theta$ via a skein relation} We have the following result from \cite{4}: 
\begin{theorem}[{\cite[Theorem~8.1]{4}}]\label{btmainthm}
Let $\mathcal{L}$ be the set of all oriented links and let $q, \lambda, E$ be indeterminates. Then the isotopy invariant of classical oriented links
\[
\Theta : \mathcal{L} \rightarrow \mathbb{C}[q^{\pm 1} , \lambda^{\pm 1}, E^{\pm 1}]
\]
can be uniquely defined by the following rules:
\begin{enumerate}
\item On crossings involving different components the following skein relation holds:
\[
\frac{1}{\sqrt \lambda}\, \Theta (L_+) - \sqrt \lambda \, \Theta (L_-) = (q - q^{-1})\, \Theta (L_0)
\]
where $L_+$, $L_-$ and $L_0$ constitute an oriented Conway triple.
\item For a  union $\mathcal{K} = \sqcup_{i=1}^r K_i$ of $r$ unlinked knots, with $r \geq 1$, it holds that:
\[
\Theta (\mathcal{K}) = E^{1-r}  P(\mathcal{K})
\]
where $P(\mathcal{K})$ is the value of the Homflypt polynomial on $\mathcal{K}$.

\end{enumerate}

\end{theorem}

\noindent Using the skein relation (1) of Theorem~\ref{btmainthm}, it was proved in \cite{4} that the invariants $\Theta_d$ distinguish a pair of $P$-equivalent links (in fact more pairs were found in \cite{4} using computational methods \cite{24}). Since the invariant $\Theta$ contains the Homflypt polynomial as well as the family of invariants $\left \{ \Theta_d \right \} $, one can now easily derive the following:

\begin{theorem}[{\cite[Theorem~8.2]{4}}]
The invariant $\Theta(q,\lambda, E)$ of classical oriented links is stronger than the Homflypt polynomial.
\end{theorem}

\begin{remark} \label{diagrem} \rm
Theorem~\ref{btmainthm} was proved in \cite{4} using the algebraic approach to $\Theta$. A self-contained skein theoretical approach to the construction of the invariant $\Theta$ is due to Louis H. Kauffman and the second author. It is implemented in \cite{27} in a more general context, where also the Kauffman (Dubrovnik) polynomial is generalized. Consequently,  new state sum models for the new invariants are defined in \cite{28}. In the notation of \cite{27} we have $\Theta(q, \lambda, E) = P[P](\zeta, a , E)$ for $a = 1 / \sqrt \lambda$ and $\zeta = q - q^{-1}$.
\end{remark}

\section{The two-variable invariant $\theta$}\label{ptlsec}

We shall now define our 2-variable generalization of the Jones polynomial, $\theta$, as claimed in the title of the paper. The invariant $\theta$ generalizes also the 1-variable invariants of classical links $\theta_d$. As for the invariant $\Theta$, different methods can be employed for defining $\theta$. We will present them all and we will show that the resulting invariants coincide. We begin with the algebraic approaches.

\subsection{{\it $\theta$ via the partition Temperley--Lieb algebra}} The Temperley--Lieb analogue for the algebra $\mathcal{E}_n(q)$ is the  partition Temperley--Lieb algebra, denoted ${\rm PTL}_{n}(q)$, which was first introduced by Juyumaya in \cite{19}. Our aim is to use this  algebra  in order to generalize the 1-variable  invariants $\theta_d$ for classical links that are derived from ${\rm FTL}_{d,n}(q)$, to a stronger 2-variable invariant. For $n\geq 3$, the  algebra ${\rm PTL}_n(q)$ is the quotient of the algebra $\mathcal{E}_n(q)$ over the ideal that is generated by the elements:
\begin{equation}\label{ptldefrel}
b_{i,i+1} := \epsilon_i \epsilon_{i+1} \left( 1+ q(b_i + b_{i+1} ) + q^2 (b_i b_{i+1} + b_{i+1} b_i ) + q^3 b_i b_{i+1} b_i \right), \quad 1\leq i \leq n-1.
\end{equation}
Thus, in terms of generators and relations, the algebra ${\rm PTL}_{d,n}(q)$ is generated by the elements $b_1, \ldots , b_{n-1}, \epsilon_1 , \ldots , \epsilon_{n-1}$, subject to the relations \eqref{btpres} plus the relation $b_{i,i+1}=0$, for $ 1 \leq i \leq n-1$. It can be easily shown \cite{19} that the defining ideal of the algebra ${\rm PTL}_{d,n}(q)$ is principal and that it is generated by the element $b_{1,2}$. Next we present the necessary and sufficient conditions so that the Markov trace $\rho$ on the algebra $\mathcal{E}_n(q)$ factors through to the quotient algebra ${\rm PTL}_n(q)$.
\begin{proposition}\label{ptlthm}
The trace $\rho$ on the algebra $\mathcal{E}_n(q)$ passes to the algebra ${\rm PTL}_n(q)$ if and only if:
\[
z = -\frac{q^{-1}E}{q^2+1} \quad \mbox{or} \quad z=-\frac{q^{-1}}{E}.
\]
\end{proposition}
In order to prove Proposition~\ref{ptlthm} we shall need the following lemma.
\begin{lemma}\label{ptllem}
The following hold for the element $b_{1,2}$ in $\mathcal{E}_n(q)$:
\[
\mathfrak{m}\, b_{1,2} = q^k b_{1,2},
\]
where $\mathfrak{m} \in \mathfrak{B}_{\mathcal{E}_3(q)}$ and $k$ is the sum of the exponents of $b_i$ in the expression of $\mathfrak{m}$. An analogous result holds for the elements $b_{1,2} \, \mathfrak{m}$, where $\mathfrak{m} \in \mathfrak{B}_{\mathcal{E}_3(q)}$.
\end{lemma}
\begin{proof}
The proof is a straightforward computation using relations \eqref{btpres}, and the definition of the elements $\epsilon_{i,j}$. We shall demonstrate here the proof for the case where $\mathfrak{m} = \epsilon_{1,3} b_1 b_2 b_1$. We have that:
\begin{align*}
\epsilon_{1,3} b_1 b_2 b_1 \, b_{1,2} &=  \epsilon_{1,3} b_1 b_2 b_1 \Big (\epsilon_1 \epsilon_2 \left (1 + q ( b_1 + b_2) + q^2( b_1 b_2 + b_2 b_1 ) + q^3 b_1 b_2 b_1 \right ) \Big)\\
& = \epsilon_{1,3} \epsilon_1 \epsilon_2 \, b_1 b_2 b_1 \left (1 + q ( b_1 + b_2) + q^2( b_1 b_2 + b_2 b_1 ) + q^3 b_1 b_2 b_1 \right ) \\
&= b_1  \epsilon_2 b_1^{-1}  \epsilon_1 \epsilon_2 \, b_1 b_2 b_1 \left (1 + q ( b_1 + b_2) + q^2( b_1 b_2 + b_2 b_1 ) + q^3 b_1 b_2 b_1 \right ) \\
 &=b_1   \epsilon_1 \epsilon_2 \, b_2 b_1 \left (1 + q ( b_1 + b_2) + q^2( b_1 b_2 + b_2 b_1 ) + q^3 b_1 b_2 b_1 \right ) \\
&=   \epsilon_1 \epsilon_2 \, b_1 b_2 b_1 \left (1 + q ( b_1 + b_2) + q^2( b_1 b_2 + b_2 b_1 ) + q^3 b_1 b_2 b_1 \right ) \\
&=  q^3 \epsilon_1 \epsilon_2 + \epsilon_1 \epsilon_2 ( q^2 + q^3 (q-q^{-1}))  (b_1 +  b_2)\\
&\quad + \epsilon_1 \epsilon_2(q +2 q^2(q-q^{-1}) + q^3(q-q^{-1})^2)   (b_1 b_2+  b_2b_1) \\
&\quad +  \epsilon_1 \epsilon_2  (1 + 2q(q-q^{-1}) + 2q^2(q-q^{-1})^2 + q^3 (q-q^{-1}) + q^3(q-q^{-1})^3) b_1  b_2 b_1\\
&=q^3 \epsilon_1 \epsilon_2 \left (1 + q ( b_1 + b_2) + q^2( b_1 b_2 + b_2 b_1 ) + q^3 b_1 b_2 b_1 \right ) = q^3 b_{1,2}.
\end{align*}

\end{proof}
\begin{proof}[Proof of Proposition~\ref{ptlthm}]
The trace $\rho$ factors through to the quotient algebra ${\rm PTL}_n(q)$ if and only if $\rho$ annihilates the defining ideal of ${\rm PTL}_n(q)$, that is if and only if $\rho(\mathfrak{m}\, b_{1,2})=0$, for $\mathfrak{m} \in \mathfrak{B}_{\mathcal{E}_n(q)}$. From the defining rules of $\rho$ one deduces that it suffices to show that $\rho$ annihilates the expressions $\mathfrak{m} \, b_{1,2}$, for $\mathfrak{m} \in \mathfrak{B}_{\mathcal{E}_3(q)}$. Using now Lemma~\ref{ptllem} we have for $\mathfrak{m} \in \mathfrak{B}_{\mathcal{E}_3(q)}$ that:
\[
\rho(\mathfrak{m}\, b_{1,2}) =0 \Leftrightarrow q^k \rho (b_{1,2}) = 0 \Leftrightarrow \rho(b_{1,2})=0.
\]
Expanding the term $b_{1,2}$ we obtain:
\begin{equation}\label{zvalptl}
\rho \left(\epsilon_1\epsilon_2 (1 + q( b_1 +b_2) + q^2 (b_1b_2 + b_2 b_1) + q^3 b_1 b_2 b_1 )\right) =0 .
\end{equation}
By the linearity of the trace $\rho$ we have that \eqref{zvalptl} is equivalent to:
\[
\left ( (q^2+1 ) qz + E \right)  \left( qz + E \right ) = 0,
\]
 which leads to the following values for $z$:
 \[
 z = -\frac{q^{-1}E}{q^2+1} \quad \mbox{or} \quad z=-\frac{q^{-1}}{E}.
 \]
\end{proof}

\begin{remark} \rm
In \cite{19} Juyumaya gave the necessary and sufficient conditions so that the Markov trace on the algebra $\mathcal{E}_n(u)$ factors through to the quotient algebra ${\rm PTL}_{n}(u)$. At that time the existence of $\rho$ wasn't proved yet (this happened later in \cite{2}), so Juyumaya conjectured its existence and used indeterminates $A$ and $B$ for the trace parameters that correspond to the elements $\widetilde b_i$ and $\widetilde \epsilon_i$ respectively (recall Remark~\ref{btswitch}). Our Proposition~\ref{ptlthm} is the analogous result for the now known trace $\rho$ in terms of the new presentation with parameter $q$.
\end{remark}

 We shall now construct a new 2-variable invariant for tied links. 
\begin{definition}\rm
Let $z=-\frac{q^{-1}E}{q^2+1}$. Substituting this value for $z$ in \eqref{lambdaval} we obtain $\lambda=q^4$ and from $\overline{\Theta}(q,\lambda, E)$ the following 2-variable invariant of {\it tied links}:
\[
\overline{\theta}(q,E) (\alpha)= \left(- \frac{q^2+1}{q\,E} \right )^{n-1} q^{2\varepsilon(\alpha)} \rho(\pi(\alpha)),
\]
where $\alpha \in TB_n$ and $\varepsilon(\alpha)$ as in \eqref{ovthet}. In analogy to \eqref{thetacl}, by restricting to classical braids we obtain a 2-variable invariant for classical links denoted by $\theta(q,E)$, which is clearly a specialization of the invariant $\Theta(q,\lambda,E)$ for $\lambda=q^4$.
 
\end{definition}
The value $z=-\frac{q^{-1}}{E}$ is discarded since basic pairs of non-isotopic links are not distinguished. 

 \subsection{\it $\theta$ as a generalization of the invariants $\theta_d$.} Let now ${\rm FTL}_{d,n}^{(\rm br)}(q)$ denote the quotient of the algebra ${\rm Y}_{d,n}^{(\rm br)}(q)$ over the two-sided ideal ${\langle e_1 e_2 g_{1,2} \rangle}$, namely:
\[
{\rm FTL}_{d,n}^{(\rm br)}(q) = \frac{ {\rm Y}_{d,n}^{(\rm br)}(q)}{{\langle e_1 e_2 g_{1,2} \rangle}}.
\] 
Note that the elements $e_i \in {\rm Y}_{d,n}^{(\rm br)} (q)$, for all $i=1, \ldots, n-1$, and thus the quotient algebra ${\rm FTL}_{d,n}^{(\rm br)}(q)$ is well-defined and it is generated only by the braiding generators $g_i$, where $i=1, \ldots ,n-1$. From the discussion in Section~\ref{brsec} we have that  ${\rm FTL}_{d,n}^{(\rm br)}(q)$  coincides with the subalgebra of ${\rm FTL}_{d,n}(q)$ that is generated by the $g_i$'s and the idempotents $e_i$, where $i=1, \ldots , n-1$. From Proposition~\ref{isoprop} now one has the following:

\begin{proposition}[{\cite[Proposition~7.4]{15}}]\label{ftlisoprop} For $d \geq n$, the Partition Temperley--Lieb algebra ${\rm PTL}_{d,n}(q)$  is isomorphic to the algebra ${\rm FTL}_{d,n}^{(br)}(q)$.
\end{proposition}

This leads to the following result  that provides the connection between the invariants $\theta_d$  and $\theta(q,E)$.

\begin{lemma}\label{trlem} For $d\geq n$, the Markov traces ${\rm tr}_{d,D}$ on ${\rm FTL}_{d,n}^{(\rm br)}(q)$ and $\rho$  on ${\rm PTL}_n(q)$ coincide when restricted to classical braids.
\end{lemma}

\begin{proof}
We know from Section~\ref{thetaclrem} that the traces ${\rm tr}_{d,D}$ and $\rho$ coincide for the case of classical braids. Moreover, from Proposition~\ref{ftlisoprop} we have that ${\rm FTL}_{d,n}^{(\rm br)}(q) \cong {\rm PTL}_n(q)$, for $d\geq n$. On the other hand, the specialized trace ${\rm tr}_{d,D}$ factors through to the quotient algebra ${\rm FTL}_{d,n}(q)$ and subsequently to ${\rm FTL}_{d,n}^{(\rm br)}(q)$ for $z = -q^{-1} E_D/(q^2 +1)$ or $z = - q^{-1} E_D$. Since we are considering only classical links, the parameter $E_D$ can be generalized to an indeterminate and, thus, $z$ coincides with the values for which the Markov trace $\rho$ factors through to ${\rm PTL}_{d,n}(q)$, as proven in Proposition~\ref{ptlthm}. This means that the traces ${\rm tr}_{d,D}$ and $\rho$ coincide also on the level of the quotient algebras  
\end{proof}

We note here that the value $z = - q^{-1} E$ is discarded in both cases for the same topological reasons mentioned above. Thus, the following corollary is immediate:

\begin{corollary}\label{invscol}
For $d\geq n$, the invariant $\theta(q,E)$ specializes to the invariants $\theta_d$ for $E= \frac{1}{d}$.
\end{corollary}

\begin{remark} \rm
From the above and following Section~\ref{thetaclrem}, for $d\geq n$, the invariant $\theta(q,E)$ can be constructed directly from ${\rm FTL}^{(\rm br)}_{d,n}(q)$.
\end{remark}

\subsection{\it $\theta$ via a skein relation}  Substituting in Theorem~\ref{btmainthm} the value  $\lambda =q^4$ corresponding to $z= q^{-1}E/(q^2+1)$ we obtain the invariant $\theta(q,E)$ as given in Theorem~\ref{mainthm}.

\begin{proof}[ Proof of Theorem~\ref{mainthm}] 
 The proof is immediate since the invariant $\theta(q,E)$ is a specialization of the 3-variable invariant $\Theta(q,\lambda,E)$ on the level of the quotient algebra ${\rm PTL}_n(q)$. Since the invariant $\Theta(q,\lambda,E)$ is well-defined \cite{4, 27}, so is $\theta(q,E)$.
\end{proof}

\begin{theorem}\label{thetstr}
The 2-variable classical link invariant $\theta(q,E)$ is stronger than the Jones polynomial.
\end{theorem}

\begin{proof}
As mentioned earlier, in \cite{4} six pairs of non-isotopic oriented classical links with the same Homflypt polynomial were found to be distinguished by the invariants $\Theta_d(q,\lambda_D)$ and $\Theta(q,\lambda, E)$. By specializing the indeterminate $\lambda = q^4$ we find that they are all still distinguished by $\theta$, namely:

{\small\begin{align*}
&\theta(L11n358\{0,1\})-\theta(L11n418\{0,0\}) = \frac{(1-E) (q-1)^5 (q+1)^5 (q^2+1)(q^2+q+1)(q^2-q+1)}{E\, q^{18}}
\\
&\theta(L11a467\{0,1\})-\theta(L11a527\{0,0\}) =  \frac{(1-E) (q-1)^5 (q+1)^5 (q^2+1)(q^2+q+1)(q^2-q+1)}{E\, q^{18}}
\\
&\theta(L11n325\{1,1\})-\theta(L11n424\{0,0\}) = \frac{(E-1) (q-1)^5 (q+1)^5 (q^2+1)(q^2+q+1)(q^2-q+1)}{E\, q^{14}}\\
&\theta(L10n79\{1,1\})-\theta(L10n95\{1,0\}) = \frac{(E-1) (q^2-1)^3 (q^8+2\,q^6+2\,q^4-1)}{E \, q^{18}}\\
&\theta(L11a404\{1,1\})-\theta(L11a428\{0,1\}) = \frac{(1-E) (q-1)^3(q+1)^3(q^2+1)(q^4+1)(q^6-q^4+1)}{E\, q^4}
\\  
&\theta(L10n76\{1,1\})-\theta(L11n425\{1,0\}) =  \frac{(E-1) (q-1)^3(q+1)^3(q^2+1)(q^4+1)}{E\, q^{10}}.\\
\end{align*}}
For $E=1$ the invariant $\theta(q,E)$ coincides with the Jones polynomial and the above six differences collapse to zero. Since the invariant $\theta(q,E)$ includes the family of invariants $\{ \theta_d \}$ as well as the Jones polynomial, we deduce the claim. \end{proof}

\begin{remark} \rm 
It is clear from the proof of Theorem~\ref{thetstr} and from the results in \cite{4} about $\Theta$ that the invariant $\theta (q,E)$ is not topologically equivalent to the Homflypt polynomial on links. Further, by Remark~\ref{remspec}, the invariant $\theta (q,E)$ is also not topologically equivalent to the Kauffman polynomial.

\end{remark}

\begin{remark}\label{ambremv} \rm
Following Remark~\ref{diagrem} one can derive the invariant $\theta(q,E)$ from the invariant $V\left [  V \right ]$. In this case, we have that $\lambda = q^4$, $a = q^{-2}$ and $\zeta = q -q^{-1}$ and  thus $V\left [ V \right ]$ coincides with the invariant $\theta(q,E)$. 
\end{remark}

\subsection{{\it A combinatorial formula for $\theta$}}

In this section we  present a closed formula for the invariant $\theta$ that involves  the Jones polynomials of sublinks of a given link. This formula is a specialization of the one that W.B.R. Lickorish proves in \cite{4} for the invariant $\Theta(q,\lambda,E)$. Namely:

\begin{theorem}[{\cite[Appendix~B]{4}}]\label{licktheta}
Let  $L$ be an oriented link with $n$ components, then:

\begin{equation}\label{clform}
\Theta(q,\lambda,E)(L) = \sum_{k=1}^n \mu^{k-1} E_k \sum_{\pi} \lambda^{\nu(\pi)} P(\pi L), 
\end{equation}
where the second summation is over all partitions  $\pi$ of the components of $L$ into $k$ (unordered) subsets and $P(\pi L)$ denotes the product of the Homflypt polynomials of the $k$ sublinks of $L$ defined by $\pi$. Furthermore, $\nu(\pi)$ is the sum of all linking numbers of pairs of components of $L$ that are distinct sets of $\pi$, $E_k = (E^{-1} -1 ) (E^{-1} -2 ) \ldots (E^{-1} - k+1)$, with $E_1=1$ and $\mu = \frac{\lambda^{-1/2} - \lambda^{1/2}}{q-q^{-1}}$.
\end{theorem}
We note that Theorem~\ref{licktheta} has been also proved independently in \cite{8} using tools from representation theory. The formula for the invariant $\theta$ is an easy corollary of Theorem~\ref{licktheta}. Indeed we have:

\begin{corollary}\label{combformula}
Let  $L$ be an oriented link with $n$ components. Then:
\[
\theta(q,E)(L) = \sum_{k=1}^n (-1)^{k-1} (q+q^{-1})^{k-1} E_k \sum_{\pi} q^{4\nu(\pi)} V(\pi L), 
\]
where $\pi$, $\nu(\pi)$, and $E_k$ are as in Theorem~\ref{licktheta}, and $V(\pi L)$ denotes the product of the Jones polynomials of the $k$ sublinks of $L$ defined by $\pi$. 
\end{corollary}

\begin{proof}
The proof follows by substituting $\lambda = q^4$ in \eqref{clform} and a straightforward computation.
\end{proof}

\subsection{{\it Relating $\theta$ to the Kauffman bracket polynomial}} 

It is clear by the involvement of the linking number in the formula for $\theta$ in Corollary~\ref{combformula} that $\theta$ depends on the orientations of the components of a link $L$ so $\theta$ cannot be directly related to the  Kauffman bracket polynomial. However, we can still relate it to its oriented extension. Starting with the well-known skein relations for the Kauffman bracket polynomial 
\begin{equation}\label{regbracket}
\langle \raisebox{-.1cm}{\begin{tikzpicture}[scale=.2]
\draw [line width=0.35mm]  (-1,-1)-- (-0.22,-0.22);
\draw  [line width=0.35mm ](-1,1)--(0,0);
\draw  [line width=0.35mm] (0.22,0.22) -- (1,1);
\draw [line width=0.35mm]   (0,0) -- +(1,-1);
\end{tikzpicture}}\, \rangle =A \langle \, \raisebox{-.07cm}{\begin{tikzpicture}[scale=.2, mydeco/.style = {decoration = {markings, 
                                                       mark = at position #1 with {\arrow{>}}}
                                       }]
\draw [line width=0.35mm] plot [smooth, tension=2] coordinates { (-1,.8) (0, 0.5) (1,.8)};
\draw [ line width=0.35mm] plot [smooth, tension=2] coordinates { (-1,-.8) (0, -0.5) (1,-.8)};
\end{tikzpicture}}\, \rangle   + A^{-1} \,\langle\, \raisebox{-.1cm}{\begin{tikzpicture}[scale=.2]
 \draw [ line width=0.35mm] plot [smooth, tension=2] coordinates { (-1,-1) (-0.3, 0) (-1,1)};
 \draw [ line width=0.35mm] plot [smooth, tension=2] coordinates { (1,-1) (0.3, 0) (1,1)};
 \end{tikzpicture}}\, \rangle 
\quad \mbox{and} \quad
 \langle \raisebox{-.1cm}{\begin{tikzpicture}[scale=.2]
\draw  [line width=0.35mm] (-1,-1)-- (0,0) ;
\draw [line width=0.35mm] (-1,1)--(-0.22,0.22);
\draw [line width=0.35mm] (0,0) -- (1,1);
\draw [line width=0.35mm]   (0.22,-0.22) -- +(.8,-.8);
\end{tikzpicture}} \, \rangle = 
A^{-1} \langle \, \raisebox{-.1cm}{\begin{tikzpicture}[scale=.2, mydeco/.style = {decoration = {markings, 
                                                       mark = at position #1 with {\arrow{>}}}
                                       }]
\draw [ line width=0.35mm] plot [smooth, tension=2] coordinates { (-1,.8) (0, 0.5) (1,.8)};
\draw [ line width=0.35mm] plot [smooth, tension=2] coordinates { (-1,-.8) (0, -0.5) (1,-.8)};
\end{tikzpicture}}\, \rangle +
 A \,\langle\, \raisebox{-.1cm}{\begin{tikzpicture}[scale=.2]
 \draw [ line width=0.35mm] plot [smooth, tension=2] coordinates { (-1,-1) (-0.3, 0) (-1,1)};
 \draw [ line width=0.35mm] plot [smooth, tension=2] coordinates { (1,-1) (0.3, 0) (1,1)};
 \end{tikzpicture}}\, \rangle     
\end{equation}
and extending to oriented links we obtain the skein rules of the {\it oriented extension of the Kauffman bracket polynomial}, which unavoidably contains conflicting orientations  (see for example \cite[Section~6]{26}):

\begin{equation}\label{clbracket}
\langle \raisebox{-.1cm}{\begin{tikzpicture}[scale=.2]
\draw [line width=0.35mm]  (-1,-1)-- (-0.22,-0.22);
\draw  [line width=0.35mm ](-1,1)--(0,0);
\draw  [line width=0.35mm] (0.22,0.22) -- (1,1)[->];
\draw [line width=0.35mm]   (0,0) -- +(1,-1)[->];
\end{tikzpicture}}\, \rangle =A \langle \, \raisebox{-.07cm}{\begin{tikzpicture}[scale=.2, mydeco/.style = {decoration = {markings, 
                                                       mark = at position #1 with {\arrow{>}}}
                                       }]
\draw [line width=0.35mm, postaction = {mydeco=.6 ,decorate}] plot [smooth, tension=2] coordinates { (-1,.8) (0, 0.5) (1,.8)};
\draw [ line width=0.35mm, postaction = {mydeco=.6 ,decorate}] plot [smooth, tension=2] coordinates { (-1,-.8) (0, -0.5) (1,-.8)};
\end{tikzpicture}}\, \rangle   + A^{-1} \,\langle\, \raisebox{-.1cm}{\begin{tikzpicture}[scale=.2]
\draw  [line width=0.35mm] (-1,-1)-- (0,0)[->] ;
\draw [line width=0.35mm] (-1,1)--(0,0)[->];
\draw [line width=0.35mm] (0.5,0) -- (1.5,1)[->];
\draw [line width=0.35mm]   (0.5,0) -- +(1,-1)[->];\end{tikzpicture}}\, \rangle 
\quad \mbox{and} \quad
 \langle \raisebox{-.1cm}{\begin{tikzpicture}[scale=.2]
\draw  [line width=0.35mm] (-1,-1)-- (0,0) ;
\draw [line width=0.35mm] (-1,1)--(-0.22,0.22);
\draw [line width=0.35mm] (0,0) -- (1,1)[->];
\draw [line width=0.35mm]   (0.22,-0.22) -- +(.8,-.8)[->];
\end{tikzpicture}} \, \rangle = 
A^{-1} \langle \, \raisebox{-.1cm}{\begin{tikzpicture}[scale=.2, mydeco/.style = {decoration = {markings, 
                                                       mark = at position #1 with {\arrow{>}}}
                                       }]
\draw [ line width=0.35mm, postaction = {mydeco=.6 ,decorate}] plot [smooth, tension=2] coordinates { (-1,.8) (0, 0.5) (1,.8)};
\draw [ line width=0.35mm, postaction = {mydeco=.6 ,decorate}] plot [smooth, tension=2] coordinates { (-1,-.8) (0, -0.5) (1,-.8)};
\end{tikzpicture}}\, \rangle +
 A \,\langle\, \raisebox{-.1cm}{\begin{tikzpicture}[scale=.2]
\draw  [line width=0.35mm] (-1,-1)-- (0,0)[->] ;
\draw [line width=0.35mm] (-1,1)--(0,0)[->];
\draw [line width=0.35mm] (0.5,0) -- (1.5,1)[->];
\draw [line width=0.35mm]   (0.5,0) -- +(1,-1)[->];\end{tikzpicture}}\, \rangle     .
\end{equation}
Subtracting by parts we obtain the following skein rule of the standard oriented bracket which is a regular isotopy invariant for oriented links:
\begin{equation}\label{brackregiso}
A \,  \langle \raisebox{-.1cm}{\begin{tikzpicture}[scale=.2]
\draw [line width=0.35mm]  (-1,-1)-- (-0.22,-0.22);
\draw  [line width=0.35mm](-1,1)--(0,0);
\draw  [line width=0.35mm] (0.22,0.22) -- (1,1)[->];
\draw [line width=0.35mm]   (0,0) -- +(1,-1)[->];
\end{tikzpicture}}\, \rangle - A^{-1} \, 
\langle  \raisebox{-.1cm}{\begin{tikzpicture}[scale=.2]
\draw  [line width=0.35mm] (-1,-1)-- (0,0) ;
\draw [line width=0.35mm] (-1,1)--(-0.22,0.22);
\draw [line width=0.35mm] (0,0) -- (1,1)[->];
\draw [line width=0.35mm]   (0.22,-0.22) -- +(.8,-.8)[->];
\end{tikzpicture}} \,  \rangle
= (A^2 - A^{-2} ) \,
\langle \, \raisebox{-.1cm}{\begin{tikzpicture}[scale=.2, mydeco/.style = {decoration = {markings, 
                                                       mark = at position #1 with {\arrow{>}}}
                                       }]
\draw [line width=0.35mm, postaction = {mydeco=.6 ,decorate}] plot [smooth, tension=2] coordinates { (-1,.8) (0, 0.5) (1,.8)};
\draw [line width=0.35mm, postaction = {mydeco=.6 ,decorate}] plot [smooth, tension=2] coordinates { (-1,-.8) (0, -0.5) (1,-.8)};
\end{tikzpicture}}\, \rangle .
\end{equation}

Let now $H(z,a)$ be the regular isotopy analogue of the Homflypt polynomial. The skein relation for $H$ is the following:

\begin{equation}\label{homregiso}
H \left (\raisebox{-.1cm}{\begin{tikzpicture}[scale=.2]
\draw [line width=0.35mm]  (-1,-1)-- (-0.22,-0.22);
\draw  [line width=0.35mm](-1,1)--(0,0);
\draw  [line width=0.35mm] (0.22,0.22) -- (1,1)[->];
\draw [line width=0.35mm]   (0,0) -- +(1,-1)[->];
\end{tikzpicture}}\, \right )-  \, 
H \big (\raisebox{-.1cm}{\begin{tikzpicture}[scale=.2]
\draw  [line width=0.35mm] (-1,-1)-- (0,0) ;
\draw [line width=0.35mm] (-1,1)--(-0.22,0.22);
\draw [line width=0.35mm] (0,0) -- (1,1)[->];
\draw [line width=0.35mm]   (0.22,-0.22) -- +(.8,-.8)[->];
\end{tikzpicture}} \,  \big )
= z \,
H \left ( \, \raisebox{-.1cm}{\begin{tikzpicture}[scale=.2, mydeco/.style = {decoration = {markings, 
                                                       mark = at position #1 with {\arrow{>}}}
                                       }]
\draw [line width=0.35mm, postaction = {mydeco=.6 ,decorate}] plot [smooth, tension=2] coordinates { (-1,.8) (0, 0.5) (1,.8)};
\draw [line width=0.35mm, postaction = {mydeco=.6 ,decorate}] plot [smooth, tension=2] coordinates { (-1,-.8) (0, -0.5) (1,-.8)};
\end{tikzpicture}}\, \right ).
\end{equation}
satisfying the following three initial conditions:
\[
 H(\bigcirc) =1, \quad
H( \raisebox{-.1cm}{
\begin{tikzpicture}[scale=.2]
\draw [line width=0.35mm]  (-.7,-.7)-- (-0.22,-0.22);
\draw  [line width=0.35mm ](-.7,.7)--(0,0);
\draw  [line width=0.35mm] (0.22,0.22) -- (.7,.7)[->];
\draw [line width=0.35mm]   (0,0) -- +(.7,-.7)[->];
 \draw [line width=0.35mm] plot [smooth, tension=2] coordinates { (-.7,.7) (0,1.3) (.7,.7)};
\end{tikzpicture}}
\ ) = a H(  
\raisebox{.06cm}{
\begin{tikzpicture}[scale=.2, mydeco/.style = {decoration = {markings, 
                                                       mark = at position #1 with {\arrow{>}}}
                                       }]
 \draw [line width=0.35mm, postaction = {mydeco=.6 ,decorate}] plot [smooth, tension=2] coordinates {(0,0) (1,.2) (2,0)};
 \end{tikzpicture}}\ ) \quad \mbox{and} \quad
 H(
\raisebox{-.1cm}{\begin{tikzpicture}[scale=.2]
\draw  [line width=0.35mm] (-.7,-.7)-- (0,0) ;
\draw [line width=0.35mm] (-.7,.7)--(-0.22,0.22);
\draw [line width=0.35mm] (0,0) -- (.7,.7)[->];
\draw [line width=0.35mm]   (0.22,-0.22) -- +(.6,-.6)[->];
 \draw [line width=0.35mm] plot [smooth, tension=2] coordinates { (-.7,.7) (0,1.3) (.7,.7)};
\end{tikzpicture}}
) = a^{-1}  H (\raisebox{.06cm}{
\begin{tikzpicture}[scale=.2, mydeco/.style = {decoration = {markings, 
                                                       mark = at position #1 with {\arrow{>}}}
                                       }]
 \draw [line width=0.35mm, postaction = {mydeco=.6 ,decorate}] plot [smooth, tension=2] coordinates {(0,0) (1,.2) (2,0)};
 \end{tikzpicture}}\ ).
\]

In order to put the oriented bracket skein rule \eqref{brackregiso} in the form of the regular Homflypt polynomial skein rule \eqref{homregiso} we define the polynomial:
\begin{equation}\label{curl}
\left \{ L \right \} : = A^{w(L)} \langle L \rangle ,
\end{equation}
where $w(L)$ is the writhe of the link diagram $L$. Note that $\left \{ \ \cdot \  \right \}$ is still a regular isotopy invariant for oriented links and it satisfies the following skein relation:

\begin{equation}\label{curlregiso}
 \left \{ \raisebox{-.1cm}{\begin{tikzpicture}[scale=.2]
\draw [line width=0.35mm]  (-1,-1)-- (-0.22,-0.22);
\draw  [line width=0.35mm](-1,1)--(0,0);
\draw  [line width=0.35mm] (0.22,0.22) -- (1,1)[->];
\draw [line width=0.35mm]   (0,0) -- +(1,-1)[->];
\end{tikzpicture}}\, \right \} - \, 
 \big \{\raisebox{-.1cm}{\begin{tikzpicture}[scale=.2]
\draw  [line width=0.35mm] (-1,-1)-- (0,0) ;
\draw [line width=0.35mm] (-1,1)--(-0.22,0.22);
\draw [line width=0.35mm] (0,0) -- (1,1)[->];
\draw [line width=0.35mm]   (0.22,-0.22) -- +(.8,-.8)[->];
\end{tikzpicture}} \,  \big\}
= (A^2-A^{-2}) \,
\left \{ \, \raisebox{-.1cm}{\begin{tikzpicture}[scale=.2, mydeco/.style = {decoration = {markings, 
                                                       mark = at position #1 with {\arrow{>}}}
                                       }]
\draw [line width=0.35mm, postaction = {mydeco=.6 ,decorate}] plot [smooth, tension=2] coordinates { (-1,.8) (0, 0.5) (1,.8)};
\draw [line width=0.35mm, postaction = {mydeco=.6 ,decorate}] plot [smooth, tension=2] coordinates { (-1,-.8) (0, -0.5) (1,-.8)};
\end{tikzpicture}}\, \right \}.
\end{equation}
which is derived from \eqref{brackregiso} and \eqref{curl} and coincides with \eqref{homregiso} for $z=(A^2- A^{-2})$. Furthermore, by defining $ f(L):=(-A^4)^{-w(L)}\left \{ L \right \}$ and by setting $q=A^{-2}$, one obtains the  ambient isotopy version of \eqref{curlregiso}, which is the skein relation of the Jones polynomial, $V,$ namely:
\begin{equation}\label{vskein}
q^{-2} V ( \raisebox{-.1cm}{\begin{tikzpicture}[scale=.2]
\draw [line width=0.35mm]  (-1,-1)-- (-0.22,-0.22);
\draw  [line width=0.35mm](-1,1)--(0,0);
\draw  [line width=0.35mm] (0.22,0.22) -- (1,1)[->];
\draw [line width=0.35mm]   (0,0) -- +(1,-1)[->];
\end{tikzpicture}}\, ) -  q^{2} \, 
V ( \raisebox{-.1cm}{\begin{tikzpicture}[scale=.2]
\draw  [line width=0.35mm] (-1,-1)-- (0,0) ;
\draw [line width=0.35mm] (-1,1)--(-0.22,0.22);
\draw [line width=0.35mm] (0,0) -- (1,1)[->];
\draw [line width=0.35mm]   (0.22,-0.22) -- +(.8,-.8)[->];
\end{tikzpicture}} \,  )
= (q-q^{-1}) \,
V ( \, \raisebox{-.1cm}{\begin{tikzpicture}[scale=.2, mydeco/.style = {decoration = {markings, 
                                                       mark = at position #1 with {\arrow{>}}}
                                       }]
\draw [line width=0.35mm, postaction = {mydeco=.6 ,decorate}] plot [smooth, tension=2] coordinates { (-1,.8) (0, 0.5) (1,.8)};
\draw [line width=0.35mm, postaction = {mydeco=.6 ,decorate}] plot [smooth, tension=2] coordinates { (-1,-.8) (0, -0.5) (1,-.8)};
\end{tikzpicture}} \, ).
\end{equation}

 Starting now from \eqref{clbracket}, considered only on crossings between different components of a link $L$, and having in mind the methods of \cite{26}, we have that:

\begin{equation}\label{brack}
\langle \raisebox{-.1cm}{\begin{tikzpicture}[scale=.2]
\draw [line width=0.35mm, draw=red]  (-1,-1)-- (-0.22,-0.22);
\draw  [line width=0.35mm, draw=black](-1,1)--(0,0);
\draw  [line width=0.35mm, draw=red] (0.22,0.22) -- (1,1)[->];
\draw [line width=0.35mm, draw=black]   (0,0) -- +(1,-1)[->];
\end{tikzpicture}}\, \rangle =A \langle \, \raisebox{-.07cm}{\begin{tikzpicture}[scale=.2, mydeco/.style = {decoration = {markings, 
                                                       mark = at position #1 with {\arrow{>}}}
                                       }]
\draw [red, line width=0.35mm, postaction = {mydeco=.6 ,decorate}] plot [smooth, tension=2] coordinates { (-1,.8) (0, 0.5) (1,.8)};
\draw [red, line width=0.35mm, postaction = {mydeco=.6 ,decorate}] plot [smooth, tension=2] coordinates { (-1,-.8) (0, -0.5) (1,-.8)};
\end{tikzpicture}}\, \rangle   + A^{-1} \,\langle\, \raisebox{-.1cm}{\begin{tikzpicture}[scale=.2]
\draw  [line width=0.35mm, draw=red] (-1,-1)-- (0,0)[->] ;
\draw [line width=0.35mm, draw=red] (-1,1)--(0,0)[->];
\draw [line width=0.35mm, draw=red] (0.5,0) -- (1.5,1)[->];
\draw [line width=0.35mm, draw=red]   (0.5,0) -- +(1,-1)[->];\end{tikzpicture}}\, \rangle 
\quad \mbox{and} \quad
 \langle \raisebox{-.1cm}{\begin{tikzpicture}[scale=.2]
\draw  [line width=0.35mm, draw=red] (-1,-1)-- (0,0) ;
\draw [line width=0.35mm, draw=black] (-1,1)--(-0.22,0.22);
\draw [line width=0.35mm, draw=red] (0,0) -- (1,1)[->];
\draw [line width=0.35mm, draw=black]   (0.22,-0.22) -- +(.8,-.8)[->];
\end{tikzpicture}} \, \rangle = 
A^{-1} \langle \, \raisebox{-.1cm}{\begin{tikzpicture}[scale=.2, mydeco/.style = {decoration = {markings, 
                                                       mark = at position #1 with {\arrow{>}}}
                                       }]
\draw [red, line width=0.35mm, postaction = {mydeco=.6 ,decorate}] plot [smooth, tension=2] coordinates { (-1,.8) (0, 0.5) (1,.8)};
\draw [red, line width=0.35mm, postaction = {mydeco=.6 ,decorate}] plot [smooth, tension=2] coordinates { (-1,-.8) (0, -0.5) (1,-.8)};
\end{tikzpicture}}\, \rangle +
 A \,\langle\, \raisebox{-.1cm}{\begin{tikzpicture}[scale=.2]
\draw  [line width=0.35mm, draw=red] (-1,-1)-- (0,0)[->] ;
\draw [line width=0.35mm, draw=red] (-1,1)--(0,0)[->];
\draw [line width=0.35mm, draw=red] (0.5,0) -- (1.5,1)[->];
\draw [line width=0.35mm, draw=red]   (0.5,0) -- +(1,-1)[->];\end{tikzpicture}}\, \rangle    ,  
\end{equation}
where different colours indicate different components of $L$. Further \eqref{brack} gives rise to  \eqref{curlregiso} and \eqref{vskein} for crossings between different components.

\smallbreak
 We then define the ambient isotopy link invariant $\left \{ \left \{ L \right \}\right \}$ of the  link diagram $L$ by the following two rules: \\
\smallbreak
\noindent (1) For a  union $\mathcal{K}^r = \sqcup_{i=1}^r K_i$, of $r$ unlinked knots with $r \geq 1$, we have that:
\begin{equation}\label{init}
\left \{ \left \{ \mathcal{K}^r\right \}\right \}= E^{1-r} V (\mathcal{K}),
\end{equation}
(2) On crossings involving different components the skein relation of the Jones polynomial holds, namely:
\begin{equation}\label{curlbrack}
q^{-2} \, \left \{  \left \{ \raisebox{-.1cm}{\begin{tikzpicture}[scale=.2]
\draw [line width=0.35mm, draw=red]  (-1,-1)-- (-0.22,-0.22);
\draw  [line width=0.35mm, draw=black](-1,1)--(0,0);
\draw  [line width=0.35mm, draw=red] (0.22,0.22) -- (1,1)[->];
\draw [line width=0.35mm, draw=black]   (0,0) -- +(1,-1)[->];
\end{tikzpicture}}\, \right \} \right \} - q^2 \, 
 \big \{  \big \{ \raisebox{-.1cm}{\begin{tikzpicture}[scale=.2]
\draw  [line width=0.35mm, draw=red] (-1,-1)-- (0,0) ;
\draw [line width=0.35mm, draw=black] (-1,1)--(-0.22,0.22);
\draw [line width=0.35mm, draw=red] (0,0) -- (1,1)[->];
\draw [line width=0.35mm, draw=black]   (0.22,-0.22) -- +(.8,-.8)[->];
\end{tikzpicture}} \,  \big\} \big\}
= (q- q^{-1} ) \,
\left \{\left \{ \, \raisebox{-.1cm}{\begin{tikzpicture}[scale=.2, mydeco/.style = {decoration = {markings, 
                                                       mark = at position #1 with {\arrow{>}}}
                                       }]
\draw [red, line width=0.35mm, postaction = {mydeco=.6 ,decorate}] plot [smooth, tension=2] coordinates { (-1,.8) (0, 0.5) (1,.8)};
\draw [red, line width=0.35mm, postaction = {mydeco=.6 ,decorate}] plot [smooth, tension=2] coordinates { (-1,-.8) (0, -0.5) (1,-.8)};
\end{tikzpicture}}\, \right \}\right \}.
\end{equation}
Comparing \eqref{init} and \eqref{curlbrack} to Remarks~\ref{ambremv} and   we deduce that $\left \{\left \{L \right \}\right \}$ coincides with the invariant $V [V]$ in the notation of \cite{27} and, thus, we recover the invariant $\theta(q,E)$.
 Hence we have proved the following:

\begin{theorem} The invariant $\theta(q,E)$ can be expressed in terms of the oriented extension of the Kauffman bracket polynomial.
\end{theorem}

\subsection{{\it Detecting the Thistlethwaite link (Louis H. Kauffman and Dimos Goundaroulis)}}

In this section we will demonstrate how the invariant $\theta(q,E)$ is able to distinguish the Thistlethwaite link \cite{33, 9} from the unlink with two components, something that the Jones polynomial is not able to see. First we shall consider the ambient isotopy version of the skein relation of the Jones polynomial applied on the Thistlethwaite link  and then we will compute the invariant $\theta(q,E)$ for this pair of links.  Kauffman and the first author made use of Bar-Natan's Knot Theory package for Mathematica \cite{7} in which the Jones polynomial is a function of $t$ and it satisfies the skein relation: 
\begin{equation}\label{mathskein}
t^{-1} V_{K_+}(t) - t V_{K_-}(t) = (t^{1/2} - t^{-1/2}) V_{K_0},
\end{equation} where $K_+,K_-,K_0$ is the oriented Conway triple.
In order to simplify computations, one may consider the following:
\begin{equation}\label{mathsubs}
  a:= t^2, z:=(t^{1/2} - t^{-1/2}),  b:= t z, c:= t^{-1} z.
  \end{equation}
   From \eqref{mathskein}, \eqref{mathsubs} we obtain the following relations:
 $$V_{K_+} = a V_{K_-} + b V_{K_0}$$ and
 $$V_{K_-} = a^{-1} V_{K_+} - c V_{K_0}.$$
It follows that the knots and links in the figures indicated in Figures~\ref{tlink} - \ref{k4} satisfy the following relation:
 $$V_{TLink} = bV_{K_1} + abV_{K_2} +-ca^2 V_{K_3} - ac V_{K_4} + V_{Unlinked}.$$

 \begin{figure}
  \centering
  \begin{minipage}[b]{0.45\textwidth}
    \includegraphics[width=\textwidth]{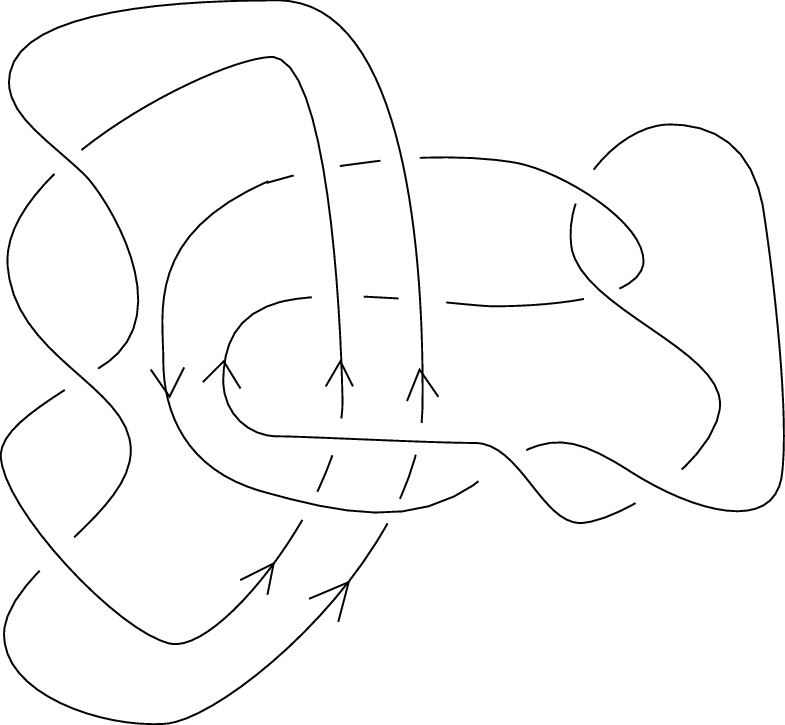}
    \caption{The Thistlethwaite link.}\label{tlink}
  \end{minipage}
  \hfill
  \begin{minipage}[b]{0.45\textwidth}
    \includegraphics[width=\textwidth]{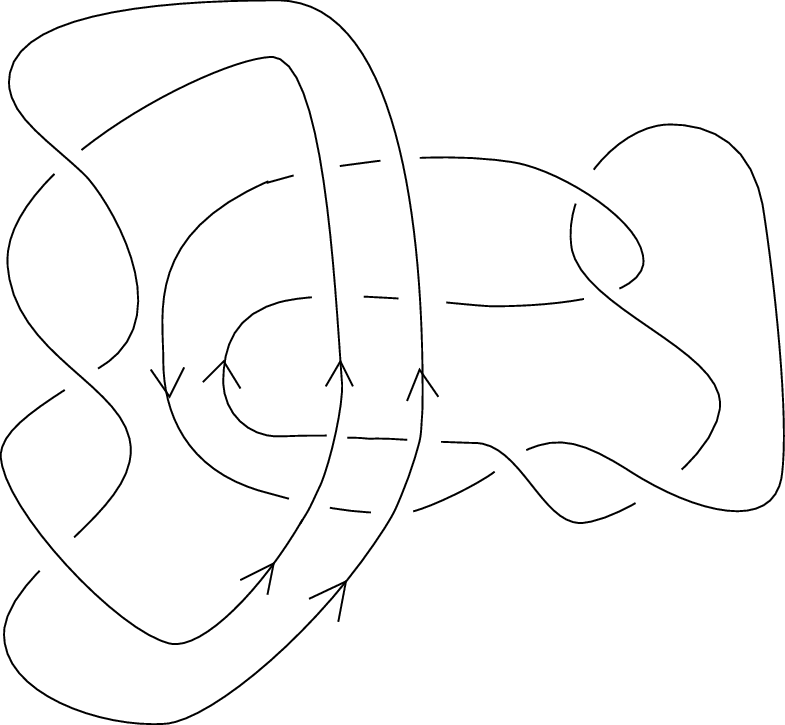}
    \caption{Unlinked link.}
  \end{minipage}

  \centering
  \begin{minipage}[b]{0.45\textwidth}
    \includegraphics[width=\textwidth]{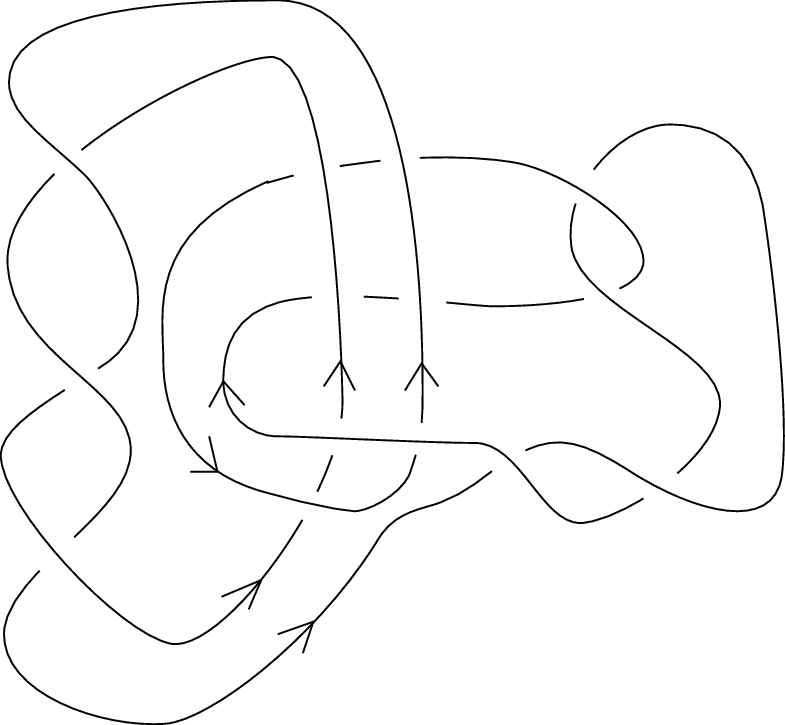}
    \caption{The knot K1.}
  \end{minipage}
  \hfill
  \begin{minipage}[b]{0.45\textwidth}
    \includegraphics[width=\textwidth]{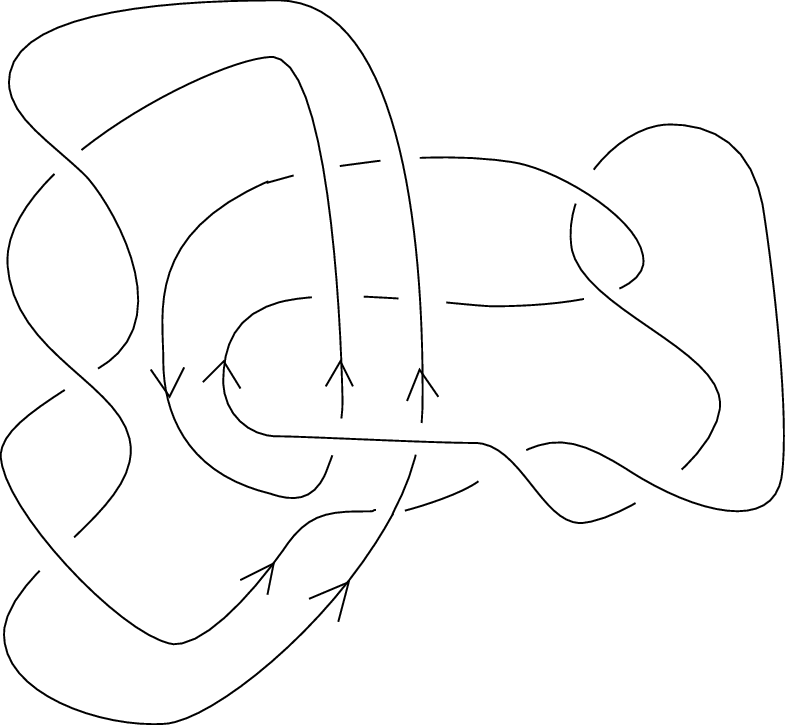}
    \caption{The knot K2.}
  \end{minipage}

  \centering
  \begin{minipage}[b]{0.45\textwidth}
    \includegraphics[width=\textwidth]{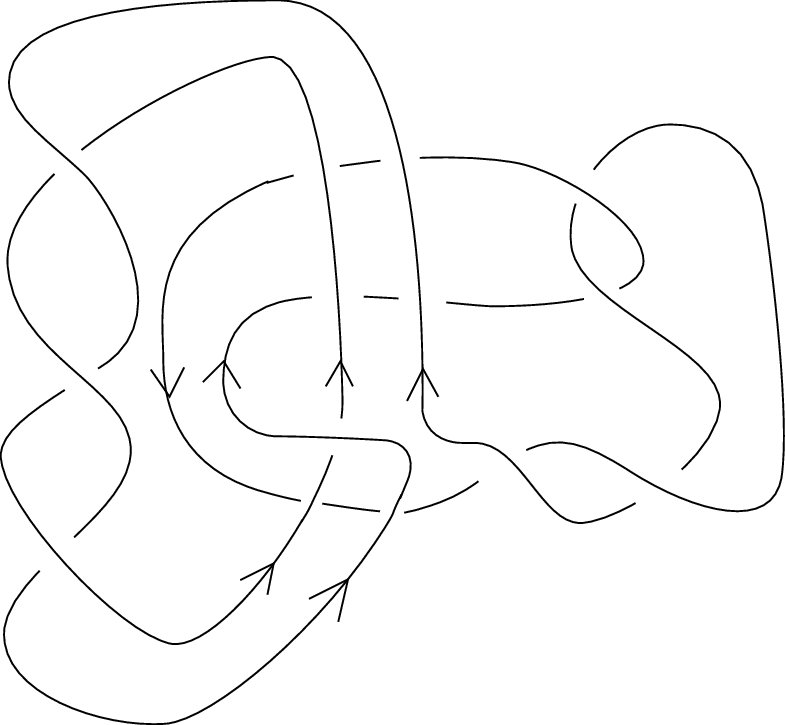}
    \caption{The knot K3.}
  \end{minipage}
  \hfill
  \begin{minipage}[b]{0.45\textwidth}
    \includegraphics[width=\textwidth]{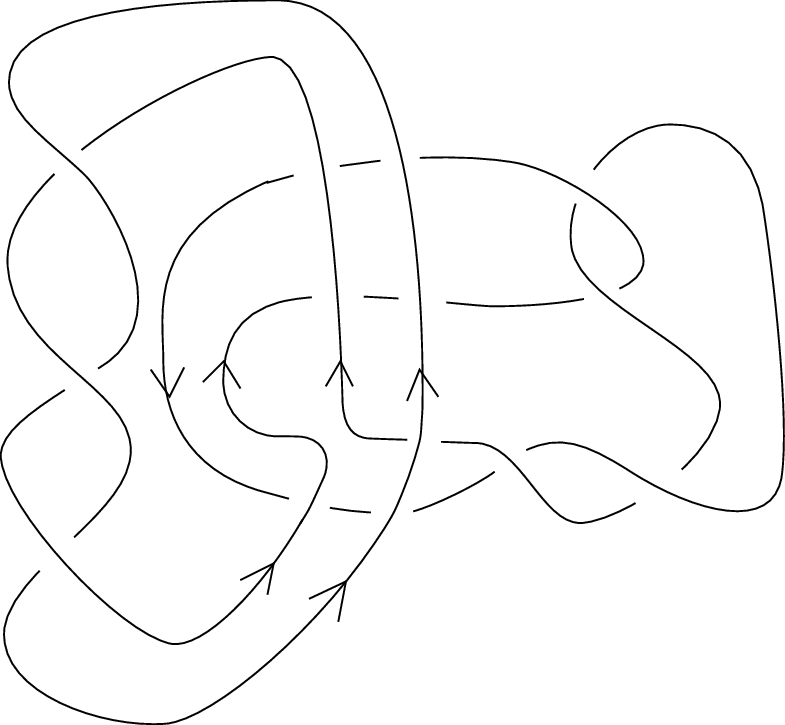}
    \caption{The knot K4.}\label{k4}
  \end{minipage}
\end{figure}
 
 This can be easily verified by the specific values computed in Mathematica:
 $$V_{TLink} = -t^{-1/2} - t^{1/2}$$
 $$V_{K_1} = -1+\frac{1}{t^7}-\frac{2}{t^6}+\frac{3}{t^5}-\frac{4}{t^4}+\frac{4}{t^3}-\frac{4}{t^2}+\frac{3}{t}+t$$
 $$V_{K_2} =  1-\frac{1}{t^9}+\frac{3}{t^8}-\frac{4}{t^7}+\frac{5}{t^6}-\frac{6}{t^5}+\frac{5}{t^4}-\frac{4}{t^3}+\frac{3}{t^2}-\frac{1}{t}$$
$$V_{K_3} = 1-\frac{1}{t^9}+\frac{2}{t^8}-\frac{3}{t^7}+\frac{4}{t^6}-\frac{4}{t^5}+\frac{4}{t^4}-\frac{3}{t^3}+\frac{2}{t^2}-\frac{1}{t}$$
$$V_{K_4} = -1-\frac{1}{t^6}+\frac{2}{t^5}-\frac{2}{t^4}+\frac{3}{t^3}-\frac{3}{t^2}+\frac{2}{t}+t$$
$$V_{Unlinked}= \frac{1}{t^{13/2}}-\frac{1}{t^{11/2}}-\frac{1}{t^{7/2}}+\frac{1}{t^{3/2}}-\frac{1}{\sqrt{t}}-t^{3/2}$$

We have proven computationally that the Thistlethwaite link cannot by distinguished from the unlink with two components using the Jones polynomial.
If one wants to compute the invariant $\theta(q,E)$, where $q=\sqrt{t}$, on the same pair of links, then the computation has to be modified to: 
$$
\theta(q,E) = bV_{K_1}(q) + abV_{K_2}(q) +-ca^2 V_{K_3}(q) - ac V_{K_4}(q) + E^{-1}V_{Unlinked}(q).$$ 
It is quite clear that this is non-trivial when the  variable $E$ is not equal to 1.

On the other hand,  the Thistlethwaite link consists of an intertwined pair of a trefoil and a figure-8 knot. Evaluating the Lickorish formula (see Corollary~\ref{combformula}) on this 2-components link yields:
$$
\theta(q,E) =(1-E^{-1})(q+q^{-1})V_{K_{3_1}}(q)V_{K_{4_1}}(q) + V_{TLink}(q).
$$
Thus the situation is completely clarified by the Lickorish formula.

%
%

\section*{Acknowledgments}
The authors would like to thank Louis H. Kauffman for fruitful conversations throughout the development of this paper. 
This research has been co-financed by the European Union (European Social
 Fund - ESF) and Greek national funds through the Operational Program
 "Education and Lifelong Learning" of the National Strategic Reference
 Framework (NSRF) - Research Funding Program: THALES: Reinforcement of the
 interdisciplinary and/or inter-institutional research and innovation, MIS: 380154.

\end{document}